\documentclass[11pt,leqno]{article}
\usepackage{graphicx, amsfonts, amsthm, amsxtra, amssymb, verbatim, makeidx}
\usepackage{subeqnarray, relsize}
\usepackage[mathscr]{euscript}
\usepackage[english]{babel}
\usepackage[fixlanguage]{babelbib}
\usepackage{tikz-cd}
\usepackage[utf8]{inputenc}
\usepackage[english]{babel}

\usepackage{wrapfig}
\usepackage{amssymb, amsmath, amsthm}
\usepackage{graphicx}
\usepackage{color}
\usepackage{amssymb}
\usepackage{url}
\usepackage{pdfpages}
\usepackage{fancyhdr}
\usepackage{subfig}
\usepackage{titlesec}
\usepackage{enumerate}
\usepackage{comment}
\usepackage{bigints}
\usepackage{diagbox}
\usepackage{cite}
\usepackage[unicode,
            psdextra,
            colorlinks=true,
            linkcolor=blue,
            citecolor=green
            ]{hyperref}
\usepackage[nameinlink,capitalize,noabbrev]{cleveref}

\makeindex
\makeglossary

\theoremstyle{definition}
\newtheorem{definition}{\bf Definition}[section]
\newtheorem{theorem}{Theorem}[section]

\newtheorem{cor}{Corollary}[section]
\newtheorem{prop}{Proposition}[section]

\newtheorem{remark}{Remark}[section]

\newtheorem{example}{Example}[section]

\renewenvironment{proof}{{\bfseries \noindent Proof} }{ \qed \\}

\begin{document}

\def\R{\mathbb{R}}                   
\def\Z{\mathbb{Z}}                   
\def\Q{\mathbb{Q}}                   
\def\C{\mathbb{C}}                   
\def\N{\mathbb{N}}                   
\def\uhp{{\mathbb H}}                
\def\A{\mathbb{A}}

\def\cf{r}
\def\T{\sf{T}}
\def\Hess{\text{Hess}}
\def\sing{\text{sing}}
\def\supp{\text{supp}}
\def\CH{{\rm CH}}
\def\P{\mathbb{P}}
\def\Gal{\text{Gal}}
\def\pr{\text{pr}}
\def\GHod{\text{GHod}}
\def\SHod{\text{SHod}}
\def\Hod{\text{Hod}}
\def\res{\text{res}}
\def\prim{\text{prim}}
\def\dR{\text{dR}}
\def\Jac{\text{Jac }}

\def\ker{{\rm ker}}              
\def\GL{{\rm GL}}                \def\HF{{\rm HF}}                
\def\ker{{\rm ker}}              
\def\coker{{\rm coker}}          
\def\im{{\rm Im}}               
\def\coim{{\rm Coim}}            

\def\End{{\rm End}}              
\def\rank{{\rm rank}}                
\def\gcd{{\rm gcd}}                  

\begin{center}
{\LARGE\bf Periods of join algebraic cycles
}
\\
\vspace{.25in} {\large {{\sc Jorge Duque Franco and Roberto Villaflor Loyola }}}
\end{center}


\begin{abstract}
We show, for all $n\ge 2$ even and $d\ge 2+\frac{4}{n}$, that the moduli of smooth degree $d$ hypersurfaces of $\mathbb{P}^{n+1}$ contains infinitely many different Hodge loci whose Zariski tangent space has the same codimension as the Hodge locus of linear cycles. We construct the Hodge cycles determining those Hodge loci as joins of $0$-dimensional cycles inside hypersurfaces of $\mathbb{P}^1$ with all their closed points defined over $\mathbb{Q}$. In order to analyze the cycle classes of these algebraic cycles, we establish a general formula for the cycle class of the join of any two algebraic cycles inside smooth hypersurfaces, expressed in terms of their periods. Furthermore, we prove that an algebraic cycle is a join of algebraic cycles, if and only if, its associated Artin Gorenstein algebra is the tensor product of the Artin Gorenstein algebras associated to each generating cycle.
\end{abstract}

\section{Introduction}
\label{intro}
Let $n\ge 2$ even and $d\ge 2+\frac{4}{n}$. Let us consider $\pi:\mathcal{X}\rightarrow U$ the family  of all smooth degree $d$ hypersurfaces of $\P^{n+1}$, where $U\subseteq H^0(\mathcal{O}_{\P^{n+1}}(d))$ is the open complement of the discriminant locus. For each $f\in U$ we denote by $X_f:=\pi^{-1}(f)=\{f=0\}\subseteq\P^{n+1}$ its corresponding hypersurface. For $f_0\in U$ fixed, let $X=X_{f_0}$. Every Hodge cycle $\delta=\delta_{f_0}\in H^{\frac{n}{2},\frac{n}{2}}(X,\Q)$ extends by parallel transport to $\delta_f\in H^n(X_f,\Q)$ for $f$ in any simply connected analytic neighbourhood of $f_0$. The \textit{Hodge locus} induced by $\delta$ is the analytic scheme supported in the germ of analytic subvariety
\begin{equation}
\label{eq:localhodge}
V_\delta:=\{f\in (U,f_0): \delta_f\in H^{\frac{n}{2},\frac{n}{2}}(X_f,\Q)\}.
\end{equation}
Using the Infinitesimal Variation of Hodge Structure (IVHS) developed by Carlson-Green-Griffiths-Harris \cite{carlson1983infinitesimal} it is well-known that $V_\delta$ is not a properly contained analytic subvariety of $(U,f_0)$, if and only if, $\delta$ is a multiple of the polarization, i.e. if its primitive part vanishes.
For non-trivial primitive Hodge cycles $\delta\in H^{\frac{n}{2},\frac{n}{2}}(X,\Q)_\prim$, it is a classical conjecture (conjectured for $n=2$ in \cite{carlson1983infinitesimal} and proved for that case by Green \cite{green1988} and Voisin \cite{voisin1988}) that the minimal codimension of $V_\delta$ is attained precisely when $\delta$ is the primitive part of a multiple of the class of a linear subvariety of $X$. In other words, the \textit{linear cycles conjecture} claims that minimal codimension Hodge loci are precisely the Hodge loci of linear cycles. This conjecture is widely open for $n\ge 4$, but asymptotically it has been established for $d\gg n$ by Otwinowska \cite{Otw02}. In the case $n=2$, the lower bound for the codimension follows from a stronger bound by Green \cite{green1984koszul}, who proved that the Zariski tangent space $T_{f_0}V_\delta$ has codimension bounded by the codimension of the Hodge locus of a linear cycle. We say that a non-trivial primitive Hodge cycle $\delta$ is a \textit{fake linear cycle} if it weakly violates the linear cycles conjecture, in the sense that its Zariski tangent space has codimension equal to the codimension of the Hodge locus of a linear cycle, but $\delta$ is not equal to the primitive part of a multiple of a linear subvariety of $X$. Our main result is the following.

\begin{theorem}
\label{thm3}
For any degree $d$ and even dimension $n$ such that $d\ge 2+\frac{4}{n}$, there are infinitely many smooth degree $d$ hypersurfaces $X$ of dimension $n$ in $\P^{n+1}$ containing infinitely many fake linear cycles in $\P(H^{\frac{n}{2},\frac{n}{2}}(X,\Q)_\prim)$.
\end{theorem}

The notion of fake linear cycles was introduced by the authors in \cite{DuqueVillaflor}, where those cycles were studied in Fermat varieties of degree $d=3,4,6$ (for all other degrees, Fermat varieties do not contain fake linear cycles by \cite{villaflor2022small} and \cite{GMCD-NL}). Our construction of fake linear cycles in hypersurfaces different from Fermat ones is based in the classical \text{join} construction for algebraic cycles.

Let us recall the join construction. Take a decomposition of $\P^{n+1}=\P(V)$ by $V=V_1\oplus V_2$, hence $\P(V_1)\cap \P(V_2)=\varnothing$. Via this decomposition we obtain a natural identification $$H^0(\mathcal{O}_{\P(V)}(1))=V^\vee= V_1^\vee\oplus V_2^\vee=H^0(\mathcal{O}_{\P(V_1)}(1))\oplus H^0(\mathcal{O}_{\P(V_2)}(1))$$
inducing the following decomposition at the level of polynomial algebras
$$
S(\P(V))=\text{Sym}(V^\vee)= \text{Sym}(V_1^\vee)\otimes\text{Sym}(V_2^\vee)=S(\P(V_1))\otimes S(\P(V_2))
$$
in other words
$$
H^0(\mathcal{O}_{\P(V)}(d))= \bigoplus_{i+j=d}H^0(\mathcal{O}_{\P(V_1)}(i))\otimes H^0(\mathcal{O}_{\P(V_2)}(j)).
$$
Let $X=\{f=0\}\subseteq\P^{n+1}$ a smooth degree $d$ hypersurface, we say $f$ is \textit{decomposable} if it lies in the image of 
$$
H^0(\mathcal{O}_{\P(V_1)}(d))\oplus H^0(\mathcal{O}_{\P(V_2)}(d))\hookrightarrow H^0(\mathcal{O}_{\P(V)}(d))
$$
for some decomposition $V=V_1\oplus V_2$, in other words, if $f=g+h$ for $g\in H^0(\mathcal{O}_{\P(V_1)}(d))$ and $h\in H^0(\mathcal{O}_{\P(V_2)}(d)).$ Assume further that $V_1$ and $V_2$ are even dimensional, say $\dim \P(V_1)=k+1$ and $\dim \P(V_2)=n-k-1$ (thus $k$ and $n$ are even), and that the hypersurfaces $X_1:=\{g=0\}\subseteq\P(V_1)$ and $X_2:=\{h=0\}\subseteq\P(V_2)$ are smooth. Given two half dimensional algebraic subvarieties $Z_1\subseteq X_1$ and  $Z_2\subseteq X_2$, their \textit{join} $J(Z_1,Z_2)\subseteq X$ is the union of all lines connecting one point of $Z_1$ with one point of $Z_2$ inside $\P(V)$. In other words, in terms of their homogeneous coordinate rings
$$
S(J(Z_1,Z_2))= S(Z_1)\otimes S(Z_2).
$$
The above definition is compatible with rational equivalence and so, it can be extended bilinearly to a map
$$
J: \CH^\frac{k}{2}(X_1)\otimes \CH^{\frac{n-k-2}{2}}(X_2)\rightarrow \CH^\frac{n}{2}(X).
$$
A natural question is to ask what is the relation between the cycle classes of $Z_1$, $Z_2$ and $J(Z_1,Z_2)$. We establish the relation between the primitive parts of each class in terms of Griffiths' basis. Recall that by Griffiths' theorem \cite[Corollary 6.12]{voisin2003hodge} we can always write the primitive part of the cycle class of an algebraic cycle $Z\in\CH^\frac{n}{2}(X)$ inside a smooth degree $d$ hypersurface $X=\{f=0\}\subseteq\P^{n+1}$ of even dimension $n$ as the $(\frac{n}{2},\frac{n}{2})$ degree (in the Hodge decomposition) projection of a residue form
\begin{equation}
[Z]_\prim=\frac{(-1)^{\frac{n}{2}+1}\frac{n}{2}!}{d}\res\left(\frac{P_Z\Omega}{f^{\frac{n}{2}+1}}\right)^{\frac{n}{2},\frac{n}{2}}
\end{equation}
where $P_Z\in \C[x_0,\ldots,x_{n+1}]_{(d-2)(\frac{n}{2}+1)}$ is some polynomial and $H^0(\Omega_{\P^{n+1}}^{n+1}(n+2))=\C\cdot\Omega$. Furthermore, the $(\frac{n}{2},\frac{n}{2})$ projection of the corresponding residue form is uniquely determined by the class of $P_Z$ in the Jacobian ring $R^f:=\C[x_0,\ldots,x_{n+1}]/J^f$, where $J^f:=\langle \frac{\partial f}{\partial x_0},\ldots,\frac{\partial f}{\partial x_{n+1}}\rangle$ is the Jacobian ideal of $f$. In accordance with the main result of \cite{villaflor2021periods} we say $P_Z\in R^f_{(d-2)(\frac{n}{2}+1)}$ is the \textit{polynomial associated} to the algebraic cycle $Z$. Using this notation we can state our second main result.

\begin{theorem}
\label{thm2}
Let $Z_1\in \CH^\frac{k}{2}(X_1)$ and $Z_2\in \CH^\frac{n-k-2}
{2}(X_2)$. Then the join $J(Z_1,Z_2)\in\CH^\frac{n}{2}(X)$ satisfies
\begin{equation}
\label{eqPjoin}
P_{J(Z_1,Z_2)}=P_{Z_1}\cdot P_{Z_2}  \hspace{0.2cm} \in R^f.  
\end{equation}
Moreover, it is possible to detect whether a primitive Hodge class $\delta\in H^{\frac{n}{2},\frac{n}{2}}(X,\Q)_\prim$ is a join of two algebraic cycles $Z_1$ and $Z_2$ in terms of their associated Artinian Gorenstein algebras $R^{g,[Z_1]}$, $R^{h,[Z_2]}$ and $R^{f,\delta}$ (see \cref{defAG}) as follows
\begin{equation}
\label{eqjoinAGalg}
R^{f,\delta}= R^{g,[Z_1]}\otimes R^{h,[Z_2]} \ \ \Longleftrightarrow \ \ \delta=c\cdot[J(Z_1,Z_2)]_\prim 
\ \ \text{ for some }c\in\Q^\times.
\end{equation}
\end{theorem}

The above result can be restated as a relation purely in terms of periods (see \cref{thm1}), this relation can be deduced as an application of a theorem of Sebastiani and Thom \cite{sebastiani1971resultat} which states that the monodromy of $g(x)+h(y):\C^{n+2}\rightarrow\C$ splits as a tensor product of the monodromies of $g:\C^{k+2}\rightarrow \C$ and $h:\C^{n-k}\rightarrow \C$. In order to obtain the relation one has to identify the monodromy invariant part of the cohomology of the affine smooth fiber with the primitive cohomology of the projective hypersurface, keeping track of the isomorphisms in homology and the compatibility with the Griffiths' bases. We remark that Sebastiani and Thom theorem relies in analytical methods. For our context, we provide an alternative algebraic proof which relies in a toric birational modification of the ambient space reducing the computation to a smooth hypersurface of a projective simplicial toric variety, and then uses tools recently developed by the second author in \cite{villaflor2023toric} to describe residue forms along hypersurfaces in toric ambient spaces. 

In the authors previous work on fake linear cycles in Fermat varieties,
the totally decomposed structure of the cohomology class of fake linear cycles (see \cite[Theorem 1.1]{DuqueVillaflor}) together with \cref{thm2} suggested that fake linear cycles can be obtained as joins of $0$-dimensional fake linear cycles inside hypersurfaces of $\P^1$. Nevertheless, the notion of fake linear cycle previously introduced (defined in terms of the codimension of the tangent space of the Hodge locus) does not make sense in dimension $0$. After finding the tensor decomposition structure described in \cref{thm2} it became apparent that the appropriate notion of fake version of any $\frac{n}{2}$-dimensional algebraic subvariety of any $n$-dimensional smooth hypersurface of $\P^{n+1}$ for $n\ge 0$ (see \cref{deffake}) should depend on the Hilbert function associated to a Hodge cycle (see \cref{defHF}) and not only of its value at $d$ (which is the one related to the Zariski tangent space of the Hodge locus). Using this new notion and the join construction we prove \cref{thm3} by constructing fake linear cycles as joins of $0$-dimensional fake linear cycles inside hypersurfaces of $\P^1$.

Note that by Otwinowska's result \cite{Otw02}, for $d\gg n$ the Hodge loci of all fake linear cycles must have codimension strictly bigger than ${\frac{n}{2}+d \choose d}-(\frac{n}{2}+1)^2$, hence each of them is either non-reduced or reduced but singular at $X$. This schematic non-smoothness can be detected using the quadratic fundamental form introduced by Maclean (see \cite{mclean2005}), which must vanish when the Hodge locus is smooth. We relate the quadratic fundamental form of the Hodge loci $V_{[Z_1]}$, $V_{[Z_2]}$ and $V_{[J(Z_1,Z_2)]}$ in \cref{corqff}. Using \cref{corqff} we show that the Hodge loci of all fake linear cycles in \cref{thm3} are non-reduced or singular at $X$ for all $d\ge 2+\frac{6}{n}$ (see \cref{thm6}).

The article is organized as follows: in \cref{sec2} we recall some preliminaries about Artinian Gorenstein ideals and the quadratic fundamental form of a Hodge locus. \cref{sec3} is devoted to the proof of \cref{thm1} which is equivalent to \cref{thm2}. In \cref{sec4} we deduce \cref{thm2} and \cref{corqff} which allows us to relate the quadratic fundamental form of the Hodge loci of two algebraic cycles and their join. \cref{sec5} we illustrate in some concrete examples some uses of \cref{corqff} by computing the Artinian Gorenstein ideal and their associated quadratic fundamental form for all combinations of two linear cycles inside Fermat varieties which are not known to have reduced Hodge loci. In \cref{sec6} we introduce the Hilbert function associated to a Hodge cycle, and use this notion in \cref{sec7} to introduce the concept of fake algebraic cycles. This section also contains the proof of \cref{thm3}.

\bigskip

\noindent\textbf{Acknowledgements.} We are grateful for the comments and corrections by the anonymous referee which helped to improve significantly the presentation of the manuscript. We are also grateful to Hossein Movasati and Remke Kloosterman for several stimulating conversations, suggestions and corrections about this work. The main part of this work was done during our visit to IMPA in the context of the ``GADEPs focused conference III: Noether-Lefschetz and Hodge loci", we want to thank to the organizers and the institution for the invitation, support and hospitality. The final preparation of this article was done during our visit to the University of Padova (UniPD). We want to thank professor Remke Kloosterman and the math department for the invitation, hospitality and support. The first author was partially supported by the Fondecyt ANID postdoctoral grant 3220631. The second author was partially supported by the Fondecyt ANID postdoctoral grant 3210020, Fondecyt ANID regular grant 1240101 and Fondecyt ANID initiation grant 11251404.

\section{Preliminaries}
\label{sec2}
\subsection{Artinian Gorenstein ideal associated to a Hodge cycle}

For the sake of completeness we will briefly recall some known facts about Artinian Gorenstein ideals associated to Hodge cycles in smooth hypersurfaces of the projective space. For a more complete exposition see \cite{villaflor2022small}.

\begin{definition}
	A graded $\C$-algebra $R$ is \textit{Artinian Gorenstein} if there exist $\sigma\in \N$ such that 
	\begin{itemize}
		\item[(i)] $R_e=0\text{ for all }e>\sigma$,
		\item[(ii)] $\dim_\C \ R_\sigma=1$,
		\item[(iii)] $\text{the multiplication map }R_i\times R_{\sigma-i}\rightarrow R_\sigma\text{ is a perfect pairing for all }i=0,\ldots,\sigma.$
	\end{itemize}
	The number $\sigma=:\text{soc}(R)$ is the \textit{socle of $R$}. We say that an ideal $I\subseteq\C[x_0,\ldots,x_{n+1}]$ is \textit{Artinian Gorenstein of socle} $\sigma=:\text{soc}(I)$ if the quotient ring $R=\C[x_0,\ldots,x_{n+1}]/I$ is Artinian Gorenstein of socle $\sigma$.
\end{definition}

The definition of the following ideal appeared first in the work of Voisin \cite{voisin89} for surfaces, and later in the work of Otwinowska \cite{Otwinowska2003} for higher dimensional varieties.

\begin{definition}
\label{defAG}
	Let $X=\{f=0\}\subseteq\P^{n+1}$ be a smooth degree $d$ hypersurface of even dimension $n$, and $\lambda\in H^{\frac{n}{2},\frac{n}{2}}(X,\Z)$ be a non-trivial Hodge cycle. Consider $J^f:=\langle \frac{\partial f}{\partial x_0},\ldots,\frac{\partial f}{\partial x_{n+1}}\rangle$ to be the Jacobian ideal, we define the \textit{Artinian Gorenstein ideal associated to $\lambda$} as
	\begin{equation}
		\label{ideal}    
		J^{f,\lambda}:=(J^f:P_\lambda),
	\end{equation}
	where $P_\lambda\in \C[x_0,\ldots,x_{n+1}]_{(d-2)(\frac{n}{2}+1)}$ is such that $\lambda_\prim=\res\left(\frac{P_\lambda\Omega}{f^{\frac{n}{2}+1}}\right)^{\frac{n}{2},\frac{n}{2}}$. This ideal is Artinian Gorenstein of $\text{soc}(J^{f,\lambda})=(d-2)(\frac{n}{2}+1)=\frac{1}{2}\text{soc}(J^f)$. We denote by $R^{f,\lambda}:=\C[x_0,\ldots,x_{n+1}]/J^{f,\lambda}$ its corresponding Artinian Gorenstein algebra.
\end{definition}

\begin{remark}
\label{rmkeqidealAG}
The importance of this ideal is that it determines, up to a rational multiple, the primitive part of the cycle class, and its Hodge locus, that is (see for instance \cite[Corollary 2.3, Remark 2.3]{villaflor2022small})

$$
J^{f,\lambda_1}=J^{f,\lambda_2} \ \ \Longleftrightarrow \ \ \exists c\in\Q^\times: (\lambda_1-c\cdot\lambda_2)_\prim=0 \ \ \Longleftrightarrow \ \ V_{\lambda_1}=V_{\lambda_2},
$$
where $V_{\lambda}$ denotes the Hodge locus associated to the class $\lambda$, as defined in \eqref{eq:localhodge}. This ideal also encodes the information of the first-order approximation of the Hodge loci in a simple way. More precisely, let $U\subseteq\C[x_0,\ldots,x_{n+1}]_d$ be the parameter space of smooth degree $d$ hypersurfaces of $\P^{n+1}$, of even dimension $n$. For $f\in U$, let $X_f=\{f=0\}\subseteq \P^{n+1}$ be the corresponding hypersurface. For every Hodge cycle $\lambda\in H^{\frac{n}{2},\frac{n}{2}}(X_f,\Z)$, we can compute the Zariski tangent space of its associated Hodge locus $V_\lambda$ as
\begin{equation}
\label{eq:2301}
T_fV_\lambda=J^{f,\lambda}_d.
\end{equation}
Where we have identified $T_fU\simeq \C[x_0,\ldots,x_{n+1}]_d$.  
\end{remark}

\subsection{Quadratic fundamental form}\label{subsecQFF}
In this section we will explore the second order invariant of the IVHS associated to the Hodge locus $V_{[Z]}$ described by Maclean \cite{mclean2005}.  This invariant allows us to derive geometric information about the Hodge locus, namely the Hodge locus is either singular or non-reduced. For this type of application see \cite{DuqueVillaflor}.

The quadratic fundamental form was described in the context of surfaces for the classical Noether-Lefschetz loci by Maclean \cite{mclean2005}. However in higher dimensions it also gives a partial description of the quadratic fundamental form.

\begin{definition}
	Let $M$ be a smooth $m$-dimensional analytic scheme, $V$ a vector bundle on $M$ and $\sigma$ a section of $V$. Let $W$ be the zero locus of $\sigma$ and let $x\in W$. The \textit{quadratic fundamental form of $\sigma$ at $x$} is
	$$
	q_{\sigma,x}:T_xW\otimes T_xW\rightarrow V_x/\text{Im}(d\sigma_x)
	$$
	given in local coordinates $(z_1,\ldots,z_m)$ around $x$ by
	$$
	q_{\sigma,x}\left(\sum_{i=1}^m\alpha_i\frac{\partial}{\partial z_i},\sum_{j=1}^m\beta_j\frac{\partial}{\partial z_j}\right)=\sum_{i=1}^m\alpha_i\frac{\partial}{\partial z_i}\left(\sum_{j=1}^m\beta_j\frac{\partial}{\partial z_j}(\sigma)\right).
	$$
\end{definition}

\begin{remark}
The quadratic fundamental form detects the second order approximation to $W$ at $x$. In particular if $W$ is smooth at $x$, then $q_{\sigma,x}$ vanishes.
\end{remark}

In our context we will take $M=(U,0)$, $V=\bigoplus_{p=0}^{\frac{n}{2}-1}\mathscr{F}^{p}/\mathscr{F}^{p+1}$ and $x=0$. Here, $U\subseteq H^0(\mathcal{O}_{\P^{n+1}}(d))$ is the parameter space of smooth degree $d$ hypersurfaces of $\P^{n+1}$, $\pi: \mathcal{X}\rightarrow U$ is the corresponding family, $\mathscr{F}^p=R^n\pi_*\Omega_{\mathcal{X}/U}^{\bullet\ge p}$, and $0\in U$ corresponds to the Fermat variety $X_d^n$. Although the next result we will use applies in a neighborhood of any point in $U$ — that is, for any smooth hypersurface in the universal family — we will restrict ourselves to a neighborhood of the Fermat variety for our computations. In order to construct a section $\sigma$ of $V$ around $x$, let $\lambda\in H^{\frac{n}{2},\frac{n}{2}}(X^n_d)_\prim\cap H^n(X^n_d,\Z)$ be a Hodge cycle, and consider $\overline{\lambda}$ its induced flat section in $\mathscr{F}^0/\mathscr{F}^\frac{n}{2}$. If we fix a holomorphic splitting $\mathscr{F}^0/\mathscr{F}^\frac{n}{2}\simeq V$ and we take $\sigma$ as the image of $\overline{\lambda}$ under this splitting, then $W=V_\lambda$. In this context we can identify $T_xW=J^{f,\lambda}_d$ \eqref{eq:2301}, $V_x=\bigoplus_{q=\frac{n}{2}+1}^n R^f_{d(q+1)-n-2}$ and $d\sigma_x=\cdot P_\lambda$. The computation of the degree $d+(d-2)(\frac{n}{2}+1)$ piece of $q=q_{\sigma,x}$ under these identifications was done by Maclean \cite[Theorem 7]{mclean2005} as follows.

\begin{theorem}[Maclean]\label{thmMaclean}
	The degree $r:=d+(d-2)(\frac{n}{2}+1)$ piece of the fundamental quadratic form is $q|_{\text{Sym}^2(J^{f,\lambda}_d)}$ where
	$$
	q: \text{Sym}^2(J^{f,\lambda})\rightarrow R^f/\langle P_\lambda\rangle
	$$
  is the bilinear form given by
		\begin{equation}
		\label{eqQFF}q(G,H)=\sum_{i=0}^{n+1}\left(H\frac{\partial Q_i}{\partial x_i}-R_i\frac{\partial G}{\partial x_i}\right)
		\end{equation}
	where 
	$$
	G\cdot P_\lambda=\sum_{i=0}^{n+1}Q_i\frac{\partial f}{\partial x_i} \ \ \ \text{ and } \ \ \ H\cdot P_\lambda=\sum_{i=0}^{n+1}R_i\frac{\partial f}{\partial x_i}.
	$$
\end{theorem}

\begin{remark}
\label{rmkqffjac}
In particular $q(\cdot,H)=0$ for any $H\in J^f$.
\end{remark}





\section{Periods of join of algebraic cycles}
\label{sec3}
In this section we compute the periods of joins of algebraic cycles. Then we use this information to relate the cycle class and Artinian Gorenstein ideals of them. Let us recall the context we are working in. 

\bigskip
We start with an odd-dimensional projective space $\P^{n+1}$ (i.e. $n$ even), and two odd-dimensional linear subspaces $\P^{k+1}, \P^{n-k-1} \subseteq \P^{n+1}$ such that $\P^{k+1} \cap \P^{n-k-1} = \varnothing$. Consider a smooth degree $d$ hypersurface $X = \{f=0\} \subseteq \P^{n+1}$, where the defining equation $f$ is \textit{decomposable}, that is, it can be written as $f = g + h$, with $g \in H^0(\mathcal{O}_{\P^{k+1}}(d))$ and $h \in H^0(\mathcal{O}_{\P^{n-k-1}}(d))$. Assume further that the degree $d$ hypersurfaces $X_1 := \{g = 0\} \subseteq \P^{k+1}$ and $X_2 := \{h = 0\} \subseteq \P^{n-k-1}$ are both smooth. Each of these hypersurfaces contains a half-dimensional algebraic cycle, denoted $Z_1 \in \CH^{\frac{k}{2}}(X_1)$ and $Z_2 \in \CH^{\frac{n-k-2}{2}}(X_2)$, respectively. We then consider their join, $J(Z_1, Z_2) \in \CH^{\frac{n}{2}}(X)$.

\begin{theorem}
\label{thm1}
For any homogeneous polynomials $P(x)\in \C[x]$ and $Q(y)\in \C[y]$ such that $\deg(P(x)\cdot Q(y))=(d-2)(\frac{n}{2}+1)$ we have
\begin{equation}
\label{eqperiodsjoin}
\frac{\frac{n}{2}!}{\frac{k}{2}!\cdot \frac{n-k-2}{2}!}\int_{J(Z_1,Z_2)}\res\left(\frac{P(x)Q(y)\Omega}{(g(x)+h(y))^{\frac{n}{2}+1}}\right)=
-2\pi i\cdot \int_{Z_1}\res\left(\frac{P\Omega'}{g^{\frac{k}{2}+1}}\right)\cdot \int_{Z_2}\res\left(\frac{Q\Omega''}{h^{\frac{n-k}{2}}}\right)
\end{equation}
if $\deg(P)=(d-2)(\frac{k}{2}+1)$ and $\deg(Q)=(d-2)(\frac{n-k}{2})$, and is zero otherwise. Where $\Omega$, $\Omega'$, and $\Omega''$ are the standard top forms of $\P^{n+1}$, $\P^{k+1}$ and $\P^{n-k-1}$ respectively.
\end{theorem}

\bigskip

\begin{proof}
In order to avoid confusion let $u=(u_0:\cdots:u_{k+1})$ be the coordinates of $\P^{k+1}$, $v=(v_0:\cdots:v_{n-k-1})$ be the coordinates of $\P^{n-k-1}$ and $(x:y)=(x_0:\cdots:x_{k+1}:y_0:\cdots:y_{n-k-1})$ be the coordinates of $\P^{n+1}$. Since \eqref{eqperiodsjoin} is independent of the choice of coordinates for $\P^{k+1}$ and $\P^{n-k-1}$, we can assume by Bertini's theorem that $X_1\cap \{u_0=0\}$ and $X_2\cap \{v_0=0\}$ are smooth hyperplane sections. By the bilinearity of \eqref{eqperiodsjoin} we can reduce ourselves to the case of monomials $P(x)=x^\alpha$ and $Q(y)=y^\beta$.
Let us treat first the case where $\deg(x^\alpha)=(d-2)(\frac{k}{2}+1)$ and $\deg(y^\beta)=(d-2)(\frac{n-k}{2})$. Let us denote
$$
\omega_{\alpha\beta}:=\res\left(\frac{x^\alpha y^\beta\Omega}{(g(x)+h(y))^{\frac{n}{2}+1}}\right)^{\frac{n}{2},\frac{n}{2}}\in H^\frac{n}{2}(X,\Omega_X^\frac{n}{2}),
$$
$$
\omega_\alpha:=\res\left(\frac{u^\alpha\Omega'}{g(u)^{\frac{k}{2}+1}}\right)^{\frac{k}{2},\frac{k}{2}}\in H^\frac{k}{2}(X_1,\Omega_{X_1}^\frac{k}{2}),
$$
$$
\omega_\beta:=\res\left(\frac{v^\beta\Omega''}{h(v)^{\frac{n-k}{2}}}\right)^{\frac{n-k-2}{2},\frac{n-k-2}{2}}\in H^{\frac{n-k-2}{2}}(X_2,\Omega_{X_2}^{\frac{n-k-2}{2}}).
$$
Consider the birational map 
$$
\varphi:\P^{k+1}\times\P^{n-k-1}\times\P^1\dasharrow \P^{n+1}
$$
$$
\varphi(u,v,t)=(t_0v_0u:t_1u_0v)
$$
whose indeterminacy locus is given by $C_0\cup C_1\cup C_2$ for
$$
C_0=\{u_0=v_0=0\}, \hspace{5mm} C_1=\{u_0=t_0=0\}, \hspace{5mm} C_2=\{v_0=t_1=0\}.
$$
Let $\P_\Sigma$ be the projective simplicial toric variety obtained by successively blowing-up $C_0$, $C_1$ and $C_2$
\[\begin{tikzcd}[column sep=small]
\P_\Sigma \arrow{r}{\pi}  \arrow{rd}{\widetilde{\varphi}} 
  & \P^{k+1}\times\P^{n-k-1}\times\P^1 \arrow[dashed]{d}{\varphi} \\
    & \P^{n+1}
\end{tikzcd}.
\]
Let us denote their Cox rings by 
\begin{align*}
S(\P^{n+1})&=\C[x_0,\ldots,x_{k+1},y_0,\ldots,y_{n-k-1}],\\
S(\P^{k+1}\times\P^{n-k-1}\times\P^1)&=\C[u_0,\ldots,u_{k+1},v_0,\ldots,v_{n-k-1},t_0,t_1],\\
S(\P_\Sigma)&=\C[a_0,\ldots,a_{k+1},b_0,\ldots,b_{n-k-1},s_0,s_1,e_0,e_1,e_2],
\end{align*}
hence we have the identifications induced by $\varphi$ and $\pi$
\begin{align*}
x_0&=t_0u_0v_0, \ \ldots \ , \ \ \ \ x_{k+1}=t_0u_{k+1}v_0,\\
y_{0}&=t_1u_0v_0, \  \ldots \ , \ y_{n-k-1}=t_1u_0v_{n-k-1},
\end{align*}
\vspace{-1.1cm}
\begin{align*}
u_0&=a_0e_0e_1, \ u_1=a_1 , \ \ldots \ , \ \ \ u_{k+1}=a_{k+1}, \\
v_0&=b_0e_0e_2, \ \ v_1=b_1, \ \ldots \ , v_{n-k-1}=b_{n-k-1},    
\end{align*}
\vspace{-0.7cm}
$$
t_0=s_0e_1, \  t_1=s_1e_2.
$$
In order to understand the fan of $\P_\Sigma$ let us write first the primitive generators of the rays corresponding to each variable. Let $M_1$, $M_2$ and $M_3$ be the character lattices of $\P^{k+1}$, $\P^{n-k-1}$ and $\P^1$ respectively. Let $N_i:=M_i^\vee$ be the dual lattice. Then, the primitive generators of the rays of $\Sigma$ belong to $N:=N_1\oplus N_2\oplus N_3$. Let us denote by $\{r^{(i)}_j\}_j$ the canonical basis of $N_i$, then the primitive generators of the rays of $\Sigma(1)$ correspond to  
$$
\rho_{a_i}=(r_i^{(1)},0,0), \ \ \rho_{a_0}=-\sum_{i=1}^{k+1}\rho_{a_i}, \ \ \rho_{b_j}=(0,r_j^{(2)},0), \ \ \rho_{b_0}=-\sum_{j=1}^{n-k-1}\rho_{b_j}, 
$$
$$
\rho_{s_1}=(0,0,r_1^{(3)}), \ \ \rho_{s_0}=-\rho_{s_1}, \ \ \rho_{e_0}=\rho_{a_0}+\rho_{b_0}, \ \ \rho_{e_1}=\rho_{a_0}+\rho_{s_0}, \ \ \rho_{e_2}=\rho_{b_0}+\rho_{s_1},    
$$
for $i=1,\ldots,k+1$ and $j=1,\ldots,n-k-1$. In order to describe the (maximal) cones of $\Sigma(n+1)$, we write the generators of its irrelevant ideal as follows
\begin{equation*}
    \begin{split}
        B(\Sigma)=&\Big\langle\big\{ a_0a_ib_0b_js_0e_1, \ a_0a_ib_js_0s_1e_1, \ a_0a_ib_js_1e_1e_2, \ a_ib_0b_js_0e_1e_2, \  a_0a_ib_0b_js_1e_2,\\
        & a_ib_0b_js_0s_1e_2, \ a_0b_0b_js_0e_0e_1, \ a_0b_js_0s_1e_0e_1, \ a_0b_js_1e_0e_1e_2, \ a_ib_0s_0e_0e_1e_2, \\ 
        & a_0a_ib_0s_1e_0e_2, \ a_ib_0s_0s_1e_0e_2, \ a_0b_0s_0e_0e_1e_2, \ a_0b_0s_1e_0e_1e_2\big\}_{\begin{smallmatrix} 1\le &i&\le& k+1 \\ 1\le &j&\le& n-k-1\end{smallmatrix}}\Big\rangle.
\end{split}
\end{equation*}
Let $Y\subseteq \P_\Sigma$ be the strict transform of $X\subseteq \P^{n+1}$ under the birational morphism $\widetilde{\varphi}$. In particular
$$
Y=\{f:=(s_0b_0)^dg(u)+(s_1a_0)^dh(v)=0\}\subseteq \P_\Sigma
$$
is a smooth hypersurface (here we use that $X_1\cap \{u_0=0\}$ and $X_2\cap \{v_0=0\}$ are smooth). 
Let $W\in \CH^\frac{n}{2}(Y)$ be the strict transform of $J(Z_1,Z_2)\in \CH^\frac{n}{2}(X)$. Since $\pi_*(W)=Z_1\times Z_2\times\P^1$, in order to obtain \eqref{eqperiodsjoin} it is enough to check that 
\begin{equation}
\label{eqperjoinW}
\frac{\frac{n}{2}!}{\frac{k}{2}!\cdot\frac{n-k-2}{2}!}\cdot\int_W\widetilde{\varphi}^*\omega_{\alpha\beta}=-\int_W\pi^*(\text{pr}_1^*\omega_\alpha\cup\text{pr}_2^*\omega_\beta\cup\text{pr}_3^*\theta)
\end{equation}
for $\theta\in H^1(\P^1,\Omega_{\P^1}^1)$ the polarization (whose period is $2\pi i$). Let 
$$
X_{1,2}:=\{g(x)=h(y)=0\}\subseteq X\subseteq \P^{n+1}
$$
which is a smooth complete intersection of bi-degree $(d,d)$, and $J(Z_1,Z_2)\in \CH^\frac{n}{2}(X_{1,2})$. Since the open sets $V_j:=\{x_j\frac{\partial g(x)}{\partial x_j}\neq 0\}$ and $V_\ell':=\{y_\ell \frac{\partial h(y)}{\partial y_\ell}\neq 0\}$ cover $X_{1,2}$ for $j=0,\ldots,k$ and $\ell=0,\ldots,n-k-2$, we can assume by the moving lemma that $J(Z_1,Z_2)$ is supported in a collection of smooth subvarieties of $X_{1,2}$ contained in $\bigcup_{j=0}^\frac{k}{2}U_j\cup\bigcup_{\ell=0}^\frac{n-k-2}{2}V_\ell$. Let us denote by $W_0\subseteq Y$ the strict transform of any of such subvarieties. Thus, in order to prove \eqref{eqperjoinW} it is enough to show that
\begin{equation}
\label{eqcohjoin}
\frac{\frac{n}{2}!}{\frac{k}{2}!\cdot\frac{n-k-2}{2}!}\cdot\widetilde{\varphi}^*\omega_{\alpha\beta}|_{W_0}=-\pi^*(\text{pr}_1^*\omega_\alpha\cup\text{pr}_2^*\omega_\beta\cup\text{pr}_3^*\theta)|_{W_0}
\end{equation}
in $H^n_\dR(W_0,\C)\simeq H^\frac{n}{2}(W_0,\Omega_{W_0}^\frac{n}{2})$. 
We can compute the left hand side of \eqref{eqcohjoin} using a toric version of a theorem due to Carlson and Griffiths \cite[Theorem 8.1]{villaflor2023toric} which computes the residue map in \v{C}ech cohomology relative to the Jacobian cover $\mathcal{U}=\{U_i\}_{i=0}^{n+6}$ of $Y$, where $U_i=\{f_i\neq0\}$ and $f_i$ are the partial derivatives of $f$ with respect to the homogeneous coordinates of $\P_\Sigma$. Let us denote by $\Omega'''$ the standard top form of $\P_\Sigma$. Since $\widetilde{\varphi}^*\Omega=-e_0e_1e_2u_0^{n-k}v_0^{k+2}t_0^{k+1}t_1^{n-k-1}\Omega'''$, we get
\begin{equation}
\label{eq:resalphabeta}
\widetilde{\varphi}^*\omega_{\alpha\beta}=\res\left(\frac{(t_0v_0)^{(d-2)(\frac{k}{2}+1)}(t_1u_0)^{(d-2)(\frac{n-k}{2})}u^\alpha v^\beta\widetilde{\varphi}^*\Omega}{(e_0e_1e_2)^{n+2}f^{\frac{n}{2}+1}}\right)^{\frac{n}{2},\frac{n}{2}}
\end{equation}
$$
=-\res\left(\frac{s_0^{d(\frac{k}{2}+1)-1}s_1^{d(\frac{n-k}{2})-1}a_0^{d(\frac{n-k}{2})}b_0^{d(\frac{k}{2}+1)}e_0u^\alpha v^\beta\Omega'''}{f^{\frac{n}{2}+1}}\right)^{\frac{n}{2},\frac{n}{2}}
$$
$$
=\frac{-1}{\frac{n}{2}!}\left\{\frac{s_0^{d(\frac{k}{2}+1)-1}s_1^{d(\frac{n-k}{2})-1}a_0^{d(\frac{n-k}{2})}b_0^{d(\frac{k}{2}+1)}e_0u^\alpha v^\beta\Omega'''_J}{f_J}\right\}_{|J|=\frac{n}{2}+1}\in H^\frac{n}{2}(\mathcal{U},\Omega_{Y}^\frac{n}{2}),
$$
where we are using the notation from \cite{villaflor2023toric}. On the other hand we have
$$
\pi^*\text{pr}_1^*\omega_\alpha|_{W_0}=\frac{1}{\frac{k}{2}!}\left\{\frac{u^\alpha\Omega'_K}{g_K}\right\}_{|K|=\frac{k}{2}+1}\in H^\frac{k}{2}(\pi^{-1}\text{pr}_1^{-1}\mathcal{U}_1,\Omega_{W_0}^\frac{k}{2}),
$$
$$
\pi^*\text{pr}_2^*\omega_\beta|_{W_0}=\frac{1}{\frac{n-k-2}{2}!}\left\{\frac{v^\beta\Omega''_L}{h_L}\right\}_{|L|=\frac{n-k}{2}}\in H^{\frac{n-k-2}{2}}(\pi^{-1}\text{pr}_2^{-1}\mathcal{U}_2,\Omega_{W_0}^{\frac{n-k-2}{2}}),
$$
$$
\pi^*\text{pr}_3^*\theta|_{W_0}=\frac{t_0dt_1-t_1dt_0}{t_0t_1}\in H^1(\pi^{-1}\text{pr}_3^{-1}\mathcal{U}_3,\Omega_{W_0}^1),
$$
where $\mathcal{U}_1$ and $\mathcal{U}_2$ are the Jacobian covers of $X_1$ and $X_2$ respectively, while $\mathcal{U}_3$ is the standard open cover of $\P^1$. When restricted to $W_0$, the coverings $\pi^{-1}\text{pr}_1^{-1}\mathcal{U}_1$, $\pi^{-1}\text{pr}_2^{-1}\mathcal{U}_2$ and $\pi^{-1}\text{pr}_3^{-1}\mathcal{U}_3$ admit a common refinement $\mathcal{V}=\{V_{(j,\ell,r)}\}_{(j,\ell,r)}=\{V_{(j,0,0)}\}_{j=0}^\frac{k}{2}\cup\{V_{(0,\ell,1)}\}_{\ell=0}^\frac{n-k-2}{2}$ where
$$
V_{(j,0,0)}=\{u_jg_j(u)v_0t_0\neq 0\}=\widetilde\varphi^{-1}V_j \ \ \ \text{ and } \ \ \ V_{(0,\ell,1)}=\{u_0v_\ell h_\ell(v)t_1\neq 0\}=\widetilde{\varphi}^{-1}V_\ell'.
$$
Hence
$$
\frac{k}{2}!\cdot\frac{n-k-2}{2}!\cdot(\pi^*\text{pr}_1^*\omega_\alpha|_{W_0}\cup\pi^*\text{pr}_2^*\omega_\beta|_{W_0}\cup\pi^*\text{pr}_3^*\theta|_{W_0})_{(j_1,\ell_1,r_1),\ldots,(j_{\frac{n}{2}+1},\ell_{\frac{n}{2}+1},r_{\frac{n}{2}+1})}=
$$
$$
(-1)^{\frac{nk}{4}+\frac{n}{2}+1}\cdot\frac{u^\alpha v^\beta\Omega'_{(j_1,\ldots,j_{\frac{k}{2}+1})}\wedge\Omega''_{(\ell_{\frac{k}{2}+1},\ldots,\ell_{\frac{n}{2}})}\wedge(t_{r_\frac{n}{2}}dt_{r_{\frac{n}{2}+1}}-t_{r_{\frac{n}{2}+1}}dt_{r_\frac{n}{2}})}{g_{(j_1,\ldots,j_{\frac{k}{2}+1})}(u)h_{(\ell_{\frac{k}{2}+1},\ldots,\ell_{\frac{n}{2}})}(v)t_{r_\frac{n}{2}}t_{r_{\frac{n}{2}+1}}}
$$
in \v{C}ech cohomology relative to the cover $\mathcal{V}$. We remark that the above formula for the cup product in \v{C}ech cohomology is well defined only for ordered tuples of indexes (the other tuples are defined by skew-symmetric extension), and the cohomology class is independent of this choice of the ordering. We will order the tuples $(j,\ell,r)$ lexicographically but with decreasing order in each entry. Then for an ordered set of tuples 
\begin{equation}
\label{eq:cech}
\frac{k}{2}!\cdot\frac{n-k-2}{2}!\cdot(\pi^*(\text{pr}_1^*\omega_\alpha\cup\text{pr}_2^*\omega_\beta\cup\text{pr}_3^*\theta)|_{W_0})_{(j_1,\ell_1,r_1),\ldots,(j_{\frac{n}{2}+1},\ell_{\frac{n}{2}+1},r_{\frac{n}{2}+1})}=
\end{equation}
$$
(-1)^{\frac{nk}{4}+\frac{n}{2}+1}\cdot\frac{u^\alpha v^\beta\pi^*(\Omega'_{(\frac{k}{2},\ldots,0)})\wedge\pi^*(\Omega''_{(\frac{n-k-2}{2},\ldots,0)})\wedge(s_1e_2d(s_0e_1)-s_0e_1d(s_1e_2))}{g_{(0,\ldots,\frac{k}{2})}(u)h_{(0,\ldots,\frac{n-k-2}{2})}(v)s_0s_1e_1e_2}
$$
if $(j_1,\ldots,j_{\frac{k}{2}+1})=(\frac{k}{2},\ldots,0)$, $(\ell_{\frac{k}{2}+1},\ldots,\ell_{\frac{n}{2}})=(\frac{n-k-2}{2},\ldots,0)$ and $(r_{\frac{n}{2}},r_{\frac{n}{2}+1})=(1,0)$, and is zero otherwise. Now it is routine to verify \eqref{eqcohjoin}  in the open covering $\mathcal{V}$ (which is a sub-covering of $\mathcal{U}|_{W_0}$) using \eqref{eq:resalphabeta} and \eqref{eq:cech}.

For the case where $\deg(x^\alpha)\neq (d-2)(\frac{k}{2}+1)$, let $r:=\deg(x^\alpha)-(d-2)(\frac{k}{2}+1)$. By the same argument as above, it is enough for us to show that 
$$
\widetilde{\varphi}^*\omega_{\alpha\beta}|_{W_0}=0\in H^\frac{n}{2}(\mathcal{V},\Omega_{W_0}^\frac{n}{2}).
$$
Using \eqref{eq:resalphabeta} in this covering we can write
$$
\widetilde{\varphi}^*\omega_{\alpha\beta}|_{W_0}=\pi^*(\pr_{12}^*\eta\cup \pr_3^*\widetilde{\theta})|_{W_0},
$$
for $\eta\in H^{\frac{n}{2}-1}(\pr_{12}(\mathcal{V}),\Omega_{X_1\times X_2}^{\frac{n}{2}-1}|_{\pr_{12}(W_0)})$ given by
$$
\eta_{(j_1,\ell_1,r_1),\ldots,(j_{\frac{n}{2}},\ell_{\frac{n}{2}},r_\frac{n}{2})}=\left(\frac{v_0}{u_0}\right)^r\cdot\frac{u^\alpha v^\beta \Omega'_{(j_1,\ldots,j_{\frac{k}{2}+1})}\wedge \Omega''_{(\ell_{\frac{k}{2}+1},\ldots,\ell_\frac{n}{2})}}{g_{(j_1,\ldots,j_{\frac{k}{2}+1})}(u)h_{(\ell_{\frac{k}{2}+1},\ldots,\ell_\frac{n}{2})}(v)},
$$
where each open set of the covering $\pr_{12}(\mathcal{V})=\{T_{(j,\ell,r)}\}=\{T_{(j,0,0)}\}_{j=0}^\frac{k}{2}\cup\{T_{(0,\ell,1)}\}_{\ell=0}^\frac{n-k-2}{2}$ is of the form $T_{(j,0,0)}=\{u_jv_0g_j(u)\neq 0\}$ or $T_{(0,\ell,1)}=\{u_0v_\ell h_\ell(v)\neq 0\}$. 
And where
$$
\widetilde{\theta}=\left(\frac{t_0}{t_1}\right)^r\theta\in H^1(\mathcal{U}_3,\Omega_{\P^1}^1).
$$
The result follows since $\widetilde{\theta}=0$ for $r\neq0$.
\end{proof}

\section{Cycle class and Hodge loci of join algebraic cycles}
\label{sec4}
In this section we translate the periods relation of \cref{thm1} into relations of the corresponding cycle classes and Hodge loci in the context of join algebraic cycles. The first relation is the content of \cref{thm2} which we prove in the following.

\bigskip

\noindent\textbf{Proof of \cref{thm2}} Applying \cite[Proposition 6.1]{villaflor2021periods} to \cref{thm1} we obtain
\begin{equation}
c=\frac{\frac{n}{2}!\cdot d}{\frac{k}{2}!\cdot \frac{n-k-2}{2}!}c_1c_2
\end{equation}
where $c,c_1,c_2\in\C^\times$ are the unique complex numbers such that
\begin{align*}
    \frac{(-1)^{\frac{n}{2}+1}\frac{n}{2}!}{d}PQP_{J(Z_1,Z_2)}&\equiv c\cdot \det(\Hess(g+h)) \ \ \text{ (mod }J^{g+h}) \\
    \frac{(-1)^{\frac{k}{2}+1}\frac{k}{2}!}{d}PP_{Z_1}&\equiv c_1\cdot \det(\Hess(g)) \ \ \text{ (mod }J^{g}) \\
    \frac{(-1)^{\frac{n-k}{2}}\frac{n-k-2}{2}!}{d}QP_{Z_2}&\equiv c_2\cdot \det(\Hess(h)) \ \ \text{ (mod }J^{h})
\end{align*}
for $P\in \C[x]_{(d-2)(\frac{k}{2}+1)}$ and $Q\in \C[y]_{(d-2)(\frac{n-k}{2})}$. Since $R^{g+h}=R^g\otimes R^h$ and $\det(\Hess(g+h))=\det(\Hess(g))\cdot \det(\Hess(h))$ it follows that
$$
PQP_{J(Z_1,Z_2)}\equiv PQP_{Z_1}P_{Z_2} \ \ \text{ (mod }J^{g+h})
$$
for all $P\in \C[x]_{(d-2)(\frac{k}{2}+1)}$ and $Q\in \C[y]_{(d-2)(\frac{n-k}{2})}$. In particular, 
\begin{equation}
\label{eqPjoin2}
x^\alpha y^\beta (P_{J(Z_1,Z_2)}-P_{Z_1}P_{Z_2})=0 \in R^{g+h}    
\end{equation}
for all monomials such that $\deg(x^\alpha)=(d-2)(\frac{k}{2}+1)$ and $\deg(y^\beta)=(d-2)(\frac{n-k}{2})$. On the other hand if $\deg(x^\alpha)>(d-2)(\frac{k}{2}+1)$ then $x^\alpha P_{Z_1}=0\in R^{g+h}$ and similarly if $\deg(y^\beta)>(d-2)(\frac{n-k}{2})$ then $y^\beta P_{Z_2}=0\in R^{g+h}$. Hence, it follows from the second part of \cref{thm1} that \eqref{eqPjoin2} holds for any monomial of degree $\deg(x^\alpha y^\beta)=(d-2)(\frac{n}{2}+1)$. Since $R^{g+h}$ is Artinian Gorenstein of socle in degree $(d-2)(n+2)$ we obtain \eqref{eqPjoin}.

Now if an element $T\in R^{g+h}_{e}$ is zero in $R^{g+h,\delta}=R^{g,[Z_1]}\otimes R^{h,[Z_2]}$ then 
$$
T=\sum_{i=0}^eT_i(x)\cdot \check{T}_{e-i}(y)
$$
where $T_i(x)\in R^g_i$, $\check{T}_{e-i}(y)\in R^h_{e-i}$ and for each $i=0,\ldots,e$, we have $T_i\in (J^g:P_{Z_1})$ or $\check{T}_{e-i}\in (J^h:P_{Z_2})$. Hence such a $T$ satisfies that $T\cdot P_{Z_1}\cdot P_{Z_2}=0\in R^{g+h}$ and so 
$$
J^{g+h,\delta}\subseteq (J^{g+h},P_{Z_1}\cdot P_{Z_2})=J^{g+h,[J(Z_1,Z_2)]}.
$$
Since both are Artinian Gorenstein ideals of socle in degree $(d-2)(\frac{n}{2}+1)$, they are equal and \eqref{eqjoinAGalg} follows from \cref{rmkeqidealAG}.$\hfill \blacksquare$

\bigskip

In \cref{subsecQFF} we recalled the quadratic fundamental form, which is a second order invariant of the Hodge loci that vanishes when the corresponding Hodge locus is smooth and reduced. As a consequence of \cref{thm2} we can relate the quadratic fundamental form of $V_{[J(Z_1,Z_2)]}$ with those of $V_{[Z_1]}$ and $V_{[Z_2]}$ as follows.

\begin{theorem}
\label{corqff}
In the same context of \cref{thm2} let us denote by
$$
q:\text{Sym}^2(J^{g+h,[J(Z_1,Z_2)]})\rightarrow R^{g+h}/\langle P_{Z_1}\cdot P_{Z_2}\rangle,
$$
$$
q_1:\text{Sym}^2(J^{g,[Z_1]})\rightarrow R^g/\langle P_{Z_1}\rangle,
$$
$$
q_2:\text{Sym}^2(J^{h,[Z_2]})\rightarrow R^h/\langle P_{Z_2}\rangle,
$$
the bilinear forms (introduced in \cref{thmMaclean}) associated to $J(Z_1,Z_2)$, $Z_1$ and $Z_2$ respectively. Consider 
$$
G=A_1(x,y)G_1(x)+A_2(x,y)G_2(y)\in J^{g+h,[J(Z_1,Z_2)]}
$$
$$
H=B_1(x,y)H_1(x)+B_2(x,y)H_2(y)\in J^{g+h,[J(Z_1,Z_2)]}$$
with $G_1,H_1\in J^{g,[Z_1]},\; G_2,H_2\in J^{h,[Z_2]}$. Then
\begin{equation}
\label{eqQFFjoin}
q(G,H)=A_1B_1P_{Z_2}q_1(G_1,H_1)+A_2B_2P_{Z_1}q_2(G_2,H_2).
\end{equation}
In consequence for any degree $e\ge 0$ we have the following:
\begin{itemize}
    \item[(i)] If $q_1$ and $q_2$ vanish in all degrees $\ell\le e$, then $q$ vanishes in degree $e$.
    \item[(ii)] If $q$ vanishes in degree $e$, then $q_1|_{\text{Sym}^2(J^{g,[Z_1]}_\ell)}\cdot \C[x]_j=0\in R^g/\langle P_{Z_1}\rangle$ for all degrees $\ell,j\ge 0$ such that $\ell\le e$, $j\le 2(e-\ell)$ and $2(e-\ell)-j\le (d-2)(\frac{n-k}{2})$. One gets a similar assertion for $q_2$ by symmetry.
\end{itemize}
\end{theorem}

\begin{proof}
Write 
$$
G_1\cdot P_{Z_1}=\sum_{i=0}^{k+1}Q_i(x)\frac{\partial g}{\partial x_i} \ \ , \ \ \ G_2\cdot P_{Z_2}=\sum_{j=0}^{n-k-1}R_j(y)\frac{\partial h}{\partial y_j} \ ,
$$
$$
H_1\cdot P_{Z_1}=\sum_{i=0}^{k+1}S_i(x)\frac{\partial g}{\partial x_i} \ \ , \ \ \ H_2\cdot P_{Z_2}=\sum_{j=0}^{n-k-1}T_j(y)\frac{\partial h}{\partial y_j} \ .
$$
then it follows by \eqref{eqQFF} that
$$
q(G,H)=A_1B_1P_{Z_2} q_1(G_1,H_1)+A_2B_2 P_{Z_1} q_2(G_2,H_2)    
$$
$$
+B_1P_{Z_2}\sum_{i=0}^{k+1}(H_1Q_i-G_1S_i)\frac{\partial A_1}{\partial x_i}+B_2P_{Z_1}\sum_{j=0}^{n-k-1}(H_2R_j-G_2T_j)\frac{\partial A_2}{\partial y_j}.
$$
Note that $\sum_{i=0}^{k+1}(H_1Q_i-G_1S_i)\frac{\partial g}{\partial x_i}=0$ and $\sum_{j=0}^{n-k-1}(H_2R_j-G_2T_j)\frac{\partial h}{\partial y_j}=0$. Since $(\frac{\partial g}{\partial x_0},\ldots,\frac{\partial g}{\partial x_{k+1}})$ and $(\frac{\partial h}{\partial y_0},\ldots,\frac{\partial h}{\partial y_{n-k-1}})$ are regular sequences, it follows by the exactness of the Koszul complex that 
$H_1Q_i-G_1S_i\in J^g$ and $H_2R_j-G_2T_j\in J^h$, and so we obtain \eqref{eqQFFjoin}. From \eqref{eqQFFjoin} we obtain (i) by a direct computation in the generators of $J^{g+h,[J(Z_1,Z_2)]}$ which are generators of either $J^{g,[Z_1]}$ or $J^{h,[Z_2]}$ (by \eqref{eqjoinAGalg}).  

In order to show (ii) consider $G=H=A(x,y)G_1(x)$ for any $\ell\le e$ and any $G_1\in J^{g,[Z_1]}_{\ell}$. Then by \eqref{eqQFFjoin}
$$
A^2\cdot P_{Z_2}\cdot q_1(G_1,G_1)=0\in R^{g+h}/\langle P_{Z_1}\cdot P_{Z_2}\rangle
$$
for all $A\in \C[x,y]_{e-\ell}$. In particular, for any monomial $x^\alpha y^\beta \in \C[x,y]_{2(e-\ell)}$ we can write it as $x^\alpha y^\beta=A_1A_2$ with $A_1,A_2\in\C[x,y]_{e-\ell}$ and so
$$
x^\alpha y^\beta \cdot P_{Z_2}\cdot q_1(G_1,G_1)=\left(\frac{(A_1+A_2)^2}{4}-\frac{(A_1-A_2)^2}{4}\right)P_{Z_2}\cdot q_1(G_1,G_1)=0\in R^{g+h}/\langle P_{Z_1}\cdot P_{Z_2}\rangle. 
$$
From this, it follows in fact that for any $j\le 2(e-\ell)$ and any two polynomials $Q(x)\in\C[x]_{j}$ and $S(y)\in\C[y]_{2(e-\ell)-j}$
$$
Q\cdot S\cdot P_{Z_2}\cdot q_1(G_1,G_1)=0\in R^{g+h}/\langle P_{Z_1}\cdot P_{Z_2}\rangle.
$$
As $2(e-\ell)-j\le (d-2)(\frac{n-k}{2})$ we choose $S$ such that $S\cdot P_{Z_2}\notin J^{h}$, then there exists some $T(x,y)\in \C[x,y]$ of bi-degree $(\ell+j,2(e-\ell)-j)$ such that
\begin{equation}
\label{eqq1}
P_{Z_2}(Q(x)\cdot S(y) \cdot q_1(G_1,G_1)-T(x,y)\cdot P_{Z_1})\in J^{g+h}.  
\end{equation}
Considering $S_1(y),\ldots, S_t(y)\in\C[y]_{2(e-\ell)-j}$ such that $\{S(y),S_1(y),\ldots,S_t(y)\}$ is a basis of $R^h_{2(e-\ell)-j}$ and $\{S_1(y),\ldots,S_p(y)\}$ is a basis of $\ker(R^h_{2(e-\ell)-j}\xrightarrow{\cdot P_{Z_2}}R^{h}_{2(e-\ell)+(d-2)(\frac{n-k}{2})-j})$ we can write 
$$
T(x,y)=U(x)S(y)+\sum_{m=1}^tU_m(x)S_m(y)\in R^{g+h}.
$$
Since $\{P_{Z_2}S(y),P_{Z_2}S_{p+1}(y),\ldots,P_{Z_2}S_t(y)\}$ is a basis of $R^h_{2(e-\ell)+(d-2)(\frac{n-k}{2})-j}$ and $R^{g+h}=R^g\otimes R^h$, then \eqref{eqq1} is equivalent to have
$$
Q(x)\cdot q_1(G_1,G_1)-U(x)\cdot P_{Z_1}=0 \in R^{g}
$$
and $P_{Z_1}U_m(x)=0\in R^g$ for all $m=p+1,\ldots, t$. Therefore $Q(x)\cdot q_1(G_1,G_1)=0\in R^g/\langle P_{Z_1}\rangle$ for all $Q(x)\in\C[x]_j$.
\end{proof}

\section{Examples in Fermat varieties}
\label{sec5}
In this section we give examples of join algebraic cycles inside Fermat varieties, illustrating how we can use the join structure to simplify their study. We focus on combinations of two linear cycles inside low degrees Fermat varieties, whose corresponding Hodge locus is not known to be reduced. This kind of combinations have already been studied by Movasati, Kloosterman and the second author \cite{MV,movasati2023hodge,villaflor2021periods,movasati2023hodge, kloosterman2023conjecture} as a non-trivial case to study the Variational Hodge Conjecture for reducible algebraic cycles. 

\bigskip

Along this section $X:=\{f:=x_0^d+\dots+x_{n+1}^d=0\}$ is the degree $d$ Fermat variety of even dimension $n$. Its automorphism group corresponds to $\text{Aut}(X)=G\rtimes \mathfrak{S}_{n+2}$, where $\mathfrak{S}_{n+2}$ acts by permutation on the coordinates and $G=(\Z/d\Z)^{n+2}/\im(a\in \Z/d\Z\mapsto (a,\ldots,a)\in (\Z/d\Z)^{n+2})\simeq (\Z/d\Z)^{n+1}$ acts diagonally as
$$
a\cdot (x_0:\cdots:x_{n+1})=(\zeta_d^{a_0}x_0:\cdots:\zeta_d^{a_{n+1}}x_{n+1}),
$$
where for any $k>0$, $\zeta_k$ denotes the $k$-th primitive root of unity $e^{\frac{2\pi i}{k}}$. The Fermat variety contains several $\frac{n}{2}$-dimensional linear cycles, which are obtained as the orbit under the action of $\text{Aut}(X)$ on the cycle
$$
\P^\frac{n}{2}:=\{x_0-\zeta_{2d}x_1=x_2-\zeta_{2d}x_3=\cdots=x_n-\zeta_{2d}x_{n+1}=0\}.
$$
\begin{example}
\label{ex:AGlincyc}
Consider the zero dimensional Fermat variety $X_0=\{x_0^d+x_1^d=0\}$, and a point $Z_0=\{(\zeta_{2d}:1)\}\subseteq X_0$. Since this is a complete intersection cycle, it follows by \cite[Theorem 1.1]{villaflor2021periods} that the cycle class of $Z_0$ has primitive part
$$
[Z_0]_\prim=\frac{-1}{d}\res\left(\frac{P_{Z_0}(x_0dx_1-x_1dx_0)}{x_0^d+x_1^d}\right)
$$
for the associated degree $(d-2)$ polynomial 
$$
P_{Z_0}=d\zeta_{2d}\left(\frac{x_0^{d-1}-(\zeta_{2d}x_1)^{d-1}}{x_0-\zeta_{2d}x_1}\right).
$$
Consequently $J^{x_0^d+x_1^d,[Z_0]}=\langle x_0-\zeta_{2d}x_1,x_1^{d-1}\rangle$ and the quadratic fundamental form $q$ vanishes (by \cref{rmkqffjac} this is reduced to check that $q(x_0-\zeta_{2d}x_1,x_0-\zeta_{2d}x_1)=0$).

For higher dimensions, the Fermat polynomial $x_0^d+\cdots+x_{n+1}^d$ can be written as a sum of $\frac{n}{2}+1$ Fermat polynomials in two variables. Let $X_i=\{x_{2i-2}^d+x_{2i-1}^d=0\}$, and $Z_i=\{(\zeta_{2d}:1)\}\subseteq X_i$ for each $i=1,\ldots,\frac{n}{2}+1$, then 
$$
\P^\frac{n}{2}=J(Z_1,\ldots,Z_{\frac{n}{2}+1}).
$$
In consequence
$$
P_{\P^\frac{n}{2}}=d^{\frac{n}{2}+1}\zeta_{2d}^{\frac{n}{2}+1}\prod_{i=1}^{\frac{n}{2}+1}\left(\frac{x_{2i-2}^{d-1}-(\zeta_{2d}x_{2i-1})^{d-1}}{x_{2i-2}-\zeta_{2d}x_{2i-1}}\right)
$$
and so $J^{f,[\P^\frac{n}{2}]}=\langle x_0-\zeta_{2d}x_1,x_1^{d-1},\ldots,x_n-\zeta_{2d}x_{n+1},x_{n+1}^{d-1}\rangle$. By item (i) of \cref{corqff} its quadratic fundamental form also vanishes. One can do similar computations for all other linear cycles in the Fermat variety.
\end{example}

\begin{example}
\label{exam:1802}
Let $-1\le m\le \frac{n}{2}$ be an integer. Consider inside $\P^{n+1}$ the linear subvarieties
$$
\P^{n-m}:=\{x_{n-2m}-\zeta_{2d}x_{n-2m+1}=x_{n-2m+2}-\zeta_{2d}x_{n-2m+3}=\cdots=x_n-\zeta_{2d}x_{n+1}=0\},
$$
$$
\P^\frac{n}{2}:=\{x_0-\zeta_{2d}x_1=x_2-\zeta_{2d}x_3=\cdots=x_{n-2m-2}-\zeta_{2d}x_{n-2m-1}=0\}\cap \P^{n-m},
$$
$$\check{\P}^{\frac{n}{2}}:=\{x_{0}-\zeta^{\alpha_0}_{2d}x_{1}=\dots=x_{n-2m-2}-\zeta_{2d}^{\alpha_{n-2m-2}}x_{n-2m-1}=0\}\cap \P^{n-m},$$
where $\alpha_0,\alpha_2,\dots,\alpha_{n-2m-2}\in\{3,5,\dots,2d-1\}$. Then

$$
\P^{m}:=\P^\frac{n}{2}\cap\check\P^\frac{n}{2}=\{x_0=x_1=x_2=x_3=\cdots=x_{n-2m-1}=0\}\cap\P^{n-m}.
$$
It turns out that the linear combination $Z:=r\P^\frac{n}{2}+\check{r}\check\P^\frac{n}{2}$ of these two $\frac{n}{2}$-dimensional linear cycles is a join algebraic cycle, for all $r,\check{r}\in\Z\setminus\{0\}.$ In fact, inside each degree $d$ Fermat variety
$$
X_1:=\{g:=x_0^d+\cdots+x_{n-2m-1}^d=0\}\subseteq\P^{n-2m-1}
$$
and 
$$X_2:=\{h:=x_{n-2m}^d+\cdots+x_{n+1}^d=0\}\subseteq\P^{2m+1}
$$
we can consider the algebraic cycles $Z_1\in \CH^{\frac{n}{2}-m-1}(X_1)$ and $Z_2\in\CH^m(X_2)$ given by
$$
Z_1:=rL+\check{r}\check{L},
$$
$$
Z_2:=\{x_{n-2m}-\zeta_{2d}x_{n-2m+1}=x_{n-2m+2}-\zeta_{2d}x_{n-2m+3}=\cdots=x_n-\zeta_{2d}x_{n+1}=0\}\subseteq X_2,
$$
where 
$$
L:=\{x_0-\zeta_{2d}x_1=x_2-\zeta_{2d}x_3=\cdots=x_{n-2m-2}-\zeta_{2d}x_{n-2m-1}=0\}\subseteq X_1,
$$
$$
\check{L}:=\{x_0-\zeta_{2d}^{\alpha_0}x_1=x_2-\zeta_{2d}^{\alpha_2}x_3=\cdots=x_{n-2m-2}-\zeta_{2d}^{\alpha_{n-2m-2}}x_{n-2m-1}=0\}\subseteq X_1.
$$
Since $\P^\frac{n}{2}=J(L,Z_2)$ and $\check\P^\frac{n}{2}=J(\check{L},Z_2)$ then 
$$
Z=r\P^\frac{n}{2}+\check{r}\check\P^\frac{n}{2}=rJ(L,Z_2)+\check{r}J(\check{L},Z_2)=J(Z_1,Z_2)\in \CH^{\frac{n}{2}}(X),
$$
where $X=\{f:=g+h=x_0^d+\cdots+x_{n+1}^d=0\}$ is the $n$-dimensional Fermat variety. By \cite[Theorem 1.3]{villaflor2021periods} the Hodge locus $V_{[Z]}$ satisfies
$$V_{[Z]}=V_{[\P^{\frac{n}{2}}]}\cap V_{[\check{\P}^{\frac{n}{2}}]}$$
whenever $d\geq 3$ and $m<\frac{n}{2}-\frac{d}{d-2}$. On the other hand, it follows from \cite[Propositions 17.8 and 17.9]{ho13} that $V_{[\P^{\frac{n}{2}}]}\cap V_{[\check{\P}^{\frac{n}{2}}]}$ is smooth and reduced without restrictions on $d$ and $m$. In particular, for $d\geq 3$ and $m<\frac{n}{2}-\frac{d}{d-2},$  $V_{[Z]}$ is smooth and reduced. The cases not covered in \cite[Theorem 1.3]{villaflor2021periods} are: $(d,m)=(3,\frac{n}{2}-3),$ in which case Movasati conjectured $V_{[Z]}$ is smooth (see \cite{movasati2023hodge}), $(d,m)=(3, \frac{n}{2}-2), (4, \frac{n}{2}-2)$ and $m=\frac{n}{2}-1$ with $r\neq \check{r}.$ In this last case when $r=\Check{r}$ the algebraic cycle $Z$ is a complete intersection and  $V_{[Z]}$ parametrizes hypersurfaces containing a
complete intersection of type $(1,1,\dots,1, 2)$. In the recent article \cite{kloosterman2023conjecture} Kloosterman showed that if $(d,m)=(3,\frac{n}{2}-3)$, $n\ge 10$ and $r\neq \check{r}$ then $V_{[Z]}$ is not smooth, disproving Movasati's conjecture. Moreover, he showed that when $r=\check{r}$ and $n\ge 4$, $V_{[Z]}$ is smooth. Similar results are obtained by Kloosterman in the cases $(d,m)=(3,\frac{n}{2}-2)$ and $(4,\frac{n}{2}-2)$. We will analyze each of the above cases separately, using the join description to determine their associated Artin Gorenstein ideals and corresponding quadratic fundamental forms.
\end{example}

\begin{prop}
	\label{prop:2401}
	Consider the notation of \cref{exam:1802}. For $d=3,\;n\geq 4,\; m=\frac{n}{2}-3,$ the Artinian Gorenstein ideal $J^{f,[Z]}$ associated to the algebraic cycle $Z$ is
\begin{equation*}
		\begin{split}
			J^{f,[Z]}=&\Big\langle \{x^2_{j}\}_{j=0}^{n+1}, \ \{x_{2j}-\zeta_6 x_{2j+1}\}_{j=3}^{\frac{n}{2}}, \ \{x_{2j}x_{2j+1}\}_{j=0}^2, \  A_1x_1x_3x_4+A_2x_1x_3x_5\\ 
			& x_{0}x_{2}+B_{1}x_{1}x_{2}+B_{2}x_{1}x_{3}, \  x_{0}x_{3}+C_{1}x_{1}x_{2}+C_{2}x_{1}x_{3}, \ x_0x_4+D_1x_1x_4+D_2x_1x_5\\
			& x_0x_5+E_1x_1x_4+E_2x_1x_5, \  x_{2}x_{4}+F_{1}x_{3}x_{4}+F_{2}x_{3}x_{5}, \ x_{2}x_{5}+G_{1}x_{3}x_{4}+G_{2}x_{3}x_{5} \Big\rangle
		\end{split}
	\end{equation*}
 where $(A_1:A_2)=(-r+\check{r}\zeta_6^{\alpha_0+\alpha_2+\alpha_4}:r\zeta_6-\check{r}\zeta_6^{\alpha_0+\alpha_2+2\alpha_4}) \in \P^1$,
 $$
B_1=\frac{-(\zeta_6^{\alpha_0+\alpha_2}-\zeta_6)}{\zeta_6^{\alpha_2}-\zeta_6}, \  B_2=\frac{\zeta_6^{\alpha_2+1}(\zeta_6^{\alpha_0}-\zeta_6)}{\zeta_6^{\alpha_2}-\zeta_6}, \ C_1=\frac{-(\zeta_6^{\alpha_0}-\zeta_6)}{\zeta_6^{\alpha_2}-\zeta_6}, \ C_2=\frac{\zeta_6(\zeta_6^{\alpha_0}-\zeta_6^{\alpha_2})}{\zeta_6^{\alpha_2}-\zeta_6},
 $$
 $$
 D_1=\frac{-(\zeta_6^{\alpha_0+\alpha_4}-\zeta_6^2)}{\zeta_6^{\alpha_4}-\zeta_6}, \ D_2=\frac{\zeta_6^{\alpha_4+1}(\zeta_6^{\alpha_0}-\zeta_6)}{\zeta_6^{\alpha_4}-\zeta_6}, \ E_1=\frac{-(\zeta_6^{\alpha_0}-\zeta_6)}{\zeta_6^{\alpha_4}-\zeta_6}, \ E_2=\frac{\zeta_6(\zeta_6^{\alpha_0}-\zeta_6^{\alpha_4})}{\zeta_6^{\alpha_4}-\zeta_6},
 $$
 $$
 F_1=\frac{-(\zeta_6^{\alpha_2+\alpha_4}-\zeta_6^2)}{\zeta_6^{\alpha_4}-\zeta_6}, \ F_2=\frac{\zeta_6^{\alpha_4+1}(\zeta_6^{\alpha_2}-\zeta_6)}{\zeta_6^{\alpha_4}-\zeta_6}, \ G_1=\frac{-(\zeta_6^{\alpha_2}-\zeta_6)}{\zeta_6^{\alpha_4}-\zeta_6}, \ G_2=\frac{\zeta_6(\zeta_6^{\alpha_2}-\zeta_6^{\alpha_4})}{\zeta_6^{\alpha_4}-\zeta_6}.
 $$
 In particular, the degree $k:=\frac{n}{2}+4$ piece of the quadratic fundamental form vanishes on $\text{Sym}^2(J^{f,[Z]}_3)$.
\end{prop}

\begin{proof}
Since $Z=J(Z_1,Z_2)$ is a join algebraic cycle, by \cref{thm2} it is enough to compute $J^{g,[Z_1]}$ and $J^{h,[Z_2]}$. In \cref{ex:AGlincyc} we already computed $J^{h,[Z_2]}$, and so we just need to show that
\begin{equation*}
\begin{split}
J^{g,[Z_1]}=&\Big\langle x_0^2, \ x_1^2, \ x_2^2, \ x_3^2, \ x_4^2, \ x_5^2, \ x_0x_1, \ x_2x_3, \ x_4x_5, \  A_1x_1x_3x_4+A_2x_1x_3x_5\\ 
			& x_{0}x_{2}+B_{1}x_{1}x_{2}+B_{2}x_{1}x_{3}, \  x_{0}x_{3}+C_{1}x_{1}x_{2}+C_{2}x_{1}x_{3}, \ x_0x_4+D_1x_1x_4+D_2x_1x_5\\
			& x_0x_5+E_1x_1x_4+E_2x_1x_5, \  x_{2}x_{4}+F_{1}x_{3}x_{4}+F_{2}x_{3}x_{5}, \ x_{2}x_{5}+G_{1}x_{3}x_{4}+G_{2}x_{3}x_{5} \Big\rangle.
   \end{split}
	\end{equation*}
Note that the right hand side is contained in the left hand side. Assume that $A_1\neq 0$ (the case where $A_1= 0$ is analogue), then the ideal generated by the leading terms in the lexicographic monomial ordering is
$$
\langle x_0^2,x_1^2,x_2^2,x_3^2,x_4^2,x_5^2, x_0x_1,x_0x_2,x_0x_3,x_0x_4,x_0x_5,x_2x_3, x_2x_4,x_2x_5,x_4x_5,x_1x_3x_4\rangle\subseteq \text{LT}(J^{g,[Z_1]}).
$$
Thus if we show that both monomial ideals have the same Hilbert function we are done (and in fact we conclude that the generators given above are a Gr\"obner basis of $J^{g,[Z_1]}$). To see this, note that for the left hand side monomial ideal, it is very easy to compute its Hilbert function, and in fact we see that the quotient ring has Hilbert function 1, 6, 6, 1, and 0 for degree bigger than 3. On the other hand the Hilbert function of $R^{g,[Z_1]}$ is of the form 1, $\ell$, $\ell$, 1 and $0$ for degree bigger than 3 (since $J^{g,Z_1}$ is Artinian Gorenstein of socle in degree 3). Thus it is reduced to show that $\ell=6$. In other words, to show that $J^{g,[Z_1]}_1=0$. And this can be shown using \cite[Proposition 2.1]{villaflor2021periods}. The last statement about the quadratic fundamental form follows from \cref{corqff} item (i) and a routine verification that the quadratic fundamental form $q_1|_{\text{Sym}^2(J^{g,[Z_1]})}$ vanishes in degree $\le 3$.
\end{proof}

\begin{prop}
In the context of \cref{exam:1802} consider $d=3$, $n\ge 2$ and $m=\frac{n}{2}-2$. The Artinian Gorenstein ideal $J^{f,[Z]}$ associated to algebraic cycle $Z$ is 
\begin{equation*}
	\begin{split}
		J^{f,[Z]}=&\Big\langle \{x^2_{j}\}_{j=0}^{2n+1}, \ \{x_{2j}-\zeta_6 x_{2j+1}\}_{j=2}^{\frac{n}{2}}, \ x_0x_1, \ x_2x_3, \ A_1x_1x_2+A_2x_1x_3, \\
  & B_1x_0x_2+B_2x_1x_2+B_3x_1x_3, \  C_1x_0x_3+C_2x_1x_2+C_3x_1x_3 \Big\rangle
	\end{split}
\end{equation*}
where $(A_1:A_2)=(r\zeta_6^2+\check{r}\zeta_6^{\alpha_0+\alpha_2}:r-\check{r}\zeta_6^{\alpha_0+2\alpha_2})\in\P^1$,
$$
(B_1:B_2:B_3)=\left\{\begin{array}{cc}
    (r\zeta_6^2+\check{r}\zeta_6^{\alpha_0+\alpha_2}:0:r\zeta_6-\check{r}\zeta_6^{2(\alpha_0+\alpha_2)}) & \text{ if }A_1\neq 0, \\
    (r-\check{r}\zeta_6^{\alpha_0+2\alpha_2}:-r\zeta_6+\check{r}\zeta_6^{2(\alpha_0+\alpha_2)}:0) & \text{ if }A_1= 0,
\end{array}\right.
$$
$$
(C_1:C_2:C_3)=\left\{\begin{array}{cc}
    (r\zeta_6^2+\check{r}\zeta_6^{\alpha_0+\alpha_2}:0:r-\check{r}\zeta_6^{2\alpha_0+\alpha_2}) & \text{ if }A_1\neq 0, \\
    (r-\check{r}\zeta_6^{\alpha_0+2\alpha_2}:-r+\check{r}\zeta_6^{2\alpha_0+\alpha_2}:0) & \text{ if }A_1= 0.
\end{array}\right.
$$
In particular, the degree $k:=\frac{n}{2}+4$ piece of the quadratic fundamental form vanishes.
\end{prop}

\begin{proof}
As in the proof of the previous proposition we just need to show that 
\begin{equation*}
	\begin{split}
		J^{g,[Z_1]}=&\Big\langle x_0^2, \ x_1^2, \ x_2^2, \ x_3^2, \ x_0x_1, \ x_2x_3, \ A_1x_1x_2+A_2x_1x_3, \\
  & B_1x_0x_2+B_2x_1x_2+B_3x_1x_3, \  C_1x_0x_3+C_2x_1x_2+C_3x_1x_3 \Big\rangle.
	\end{split}
\end{equation*}
The right hand side ideal is clearly contained in $J^{g,[Z_1]}$, hence it is enough to show that both ideals have the same Hilbert function. If $A_1\neq 0$ then the leading terms ideal of the right hand side ideal is
$$
\langle x_0^2, \ x_1^2, \ x_2^2, \ x_3^2, \ x_0x_1, \ x_2x_3, \ x_1x_2, \ x_0x_2, \ x_0x_3\rangle 
$$
whose quotient ring has Hilbert function equal to 1, 4, 1 and $0$ for degree bigger than $2$. If $A_1= 0$, the leading terms ideal is
$$
\langle x_0^2, \ x_1^2, \ x_2^2, \ x_3^2, \ x_0x_1, \ x_2x_3, \ x_1x_3, \ x_0x_2, \ x_0x_3\rangle 
$$
whose quotient ring also has Hilbert function equal to 1, 4, 1 and $0$ for degree bigger than $2$. Thus we are reduced to show that $R^{g,[Z_1]}$ has the same Hilbert function. Since $J^{g,[Z_1]}$ is Artinian Gorenstein of socle in degree 2 we just need to show that $J^{g,[Z_1]}_1=0$, which follows from \cite[Proposition 2.1]{villaflor2021periods}. The statement about the quadratic fundamental form follows from \cref{corqff} item (i) and the fact that $q_1|_{\text{Sym}^2(J^{g,[Z_1]})}$ vanishes in degree $\le 2$.
\end{proof}

\begin{prop}
In the context of \cref{exam:1802} let $d=4$, $n\ge 2$ and $m=\frac{n}{2}-2$. The Artinian Gorenstein ideal $J^{f,[Z]}$ associated to the algebraic cycle $Z$ is
	\begin{equation*}
	\begin{split}
		J^{f,[Z]}=&\Big\langle \{x^3_{2j+1}\}_{j=0}^{\frac{n}{2}}, \ \{x_{2j}-\zeta_8 x_{2j+1}\}_{j=2}^{\frac{n}{2}}, \ x_0x_1^2, \ x_2x_3^2, \ A_1x_1^2x_2x_3+A_2x_1^2x_3^2, \\ 
		& x_0x_2+B_1x_1x_2+B_2x_1x_3, \ x_0x_3+C_1x_1x_2+C_2x_1x_3, \\
  & x_0^2+D_1x_0x_1+D_2x_1^2, \ x_2^2+E_1x_2x_3+E_2x_3^2\Big\rangle
	\end{split}
\end{equation*}
where $(A_1:A_2)=(r\zeta_8^{2}+\check{r}\zeta_8^{\alpha_0+\alpha_2}:-(r\zeta_8^{3}+\check{r}\zeta_8^{\alpha_0+2\alpha_2}))\in\P^1$,
$$
B_1=\frac{\zeta_8^2-\zeta_8^{\alpha_0+\alpha_2}}{\zeta_8^{\alpha_2}-\zeta_8}, \ B_2=\frac{\zeta_8(\zeta_8^{\alpha_0+\alpha_2}-\zeta_8^{\alpha_2+1})}{\zeta_8^{\alpha_2}-\zeta_8}, \ C_1=\frac{\zeta_8-\zeta_8^{\alpha_0}}{\zeta_8^{\alpha_2}-\zeta_8}, \ C_2=\frac{\zeta_8(\zeta_8^{\alpha_0}-\zeta_8^{\alpha_2})}{\zeta_8^{\alpha_2}-\zeta_8},
$$
$$
D_1=\frac{-(\zeta_8^{2(\alpha_0+1)}+1)}{\zeta_8^{2}(\zeta_8^{\alpha_0}-\zeta_8)}, \ D_2=\frac{\zeta_8^{\alpha_0}(1+\zeta_8^{\alpha_0+3})}{\zeta_8^{2}(\zeta_8^{\alpha_0}-\zeta_8)}, \ E_1=\frac{-(\zeta_8^{2(\alpha_2+1)}+1)}{\zeta_8^{2}(\zeta_8^{\alpha_2}-\zeta_8)}, \ E_2=\frac{\zeta_8^{\alpha_2}(1+\zeta_8^{\alpha_2+3})}{\zeta_8^{2}(\zeta_8^{\alpha_2}-\zeta_8)}.
$$
In particular, the degree $k:=n+6$ piece of the quadratic fundamental form vanishes.
\end{prop}

\begin{proof}
As in the other cases, we are reduced to show that
\begin{equation*}
	\begin{split}
		J^{g,[Z_1]}=&\Big\langle x_1^3, \ x_3^3, \ x_0x_1^2, \ x_2x_3^2, \ A_1x_1^2x_2x_3+A_2x_1^2x_3^2, \\ 
		& x_0x_2+B_1x_1x_2+B_2x_1x_3, \ x_0x_3+C_1x_1x_2+C_2x_1x_3, \\
  & x_0^2+D_1x_0x_1+D_2x_1^2, \ x_2^2+E_1x_2x_3+E_2x_3^2\Big\rangle.
	\end{split}
 \end{equation*}
The right hand side ideal is clearly contained in $J^{g,[Z_1]}$. Let us assume that $A_1\neq 0$ (the case $A_1= 0$ is analogue), taking the ideal generated by the leading terms
in the lexicographical monomial ordering we get
$$
\langle x_0^2, \ x_2^2, \ x_0x_2, \ x_0x_3, \ x_1^3, \ x_3^3, \ x_0x_1^2, \ x_2x_3^2, \ x_1^2x_2x_3\rangle\subseteq \text{LT}(J^{g,Z_1}).
$$
For the left monomial ideal, the quotient ring has Hilbert function $1,4,6,4,1$ and $0$ for degree bigger than $4$. Thus it is enough to show that $J^{g,[Z_1]}_1=0$ and $\dim \  J^{g,[Z_1]}_2=4$. For this we use again \cite[Proposition 2.1]{villaflor2021periods} and check that 
$$
J^{g,[Z_1]}_1=\langle x_0-\zeta_8x_1,x_2-\zeta_8x_3\rangle_1\cap\langle x_0-\zeta_8^{\alpha_0}x_1,x_2-\zeta_8^{\alpha_2}x_3\rangle_1=0
$$
and
$$
J^{g,[Z_1]}_2=\langle x_0-\zeta_8x_1,x_2-\zeta_8x_3\rangle_2\cap\langle x_0-\zeta_8^{\alpha_0}x_1,x_2-\zeta_8^{\alpha_2}x_3\rangle_2
$$
$$
=\langle (x_0-\zeta_8x_1)(x_0-\zeta_8^{\alpha_0}x_1),(x_2-\zeta_8x_3)(x_0-\zeta_8^{\alpha_0}x_1),(x_0-\zeta_8x_1)(x_2-\zeta_8^{\alpha_2}x_3),(x_2-\zeta_8x_3)(x_2-\zeta_8^{\alpha_2}x_3)\rangle_2.
$$
The statement about the quadratic fundamental forms follows from \cref{corqff} item (i) and the verification that the quadratic fundamental form $q_1|_{\text{Sym}^2(J^{g,[Z_1]})}$ vanishes in degree $\le 4$.
\end{proof}

With the same notation of \cref{exam:1802}, the last remaining case to analyze is $m=\frac{n}{2}-1$ and $r\neq \check{r}$. In this case we are not giving a full description of the generators of the Artinian Gorenstein ideal, but instead we provide an element where the quadratic fundamental never vanishes for $d\ge 2+\frac{8}{n}$.

\begin{theorem}
\label{thm4}
For $d\ge 2+\frac{8}{n}$, $m=\frac{n}{2}-1$ and $r\neq\check{r}$, the degree $k=d+(d-2)(\frac{n}{2}+1)$ piece of the quadratic fundamental form associated to the Hodge locus $V_{r[\P^\frac{n}{2}]+\check{r}[\check{\P}^\frac{n}{2}]}$ does not vanish at the Fermat point. In consequence $V_{r[\P^\frac{n}{2}]+\check{r}[\check{\P}^\frac{n}{2}]}$ is not smooth.
\end{theorem}

\begin{proof}
The condition $d\ge 2+\frac{8}{n}$ allows use to use \cref{corqff} item (ii) applied for $e=d$, $\ell=d-2$ and $j=k=0$, to reduce the theorem to show that $q_1|_{\text{Sym}^2(J^{g,[Z_1]})}$ is non-zero in degree $d-2$. Consider
$$
G:=x_1^{d-3}((r\zeta_{2d}+\check{r}\zeta_{2d}^{\alpha_0})x_0+(r\zeta_{2d}^2+\check{r}\zeta_{2d}^{2\alpha_0})x_1)\in J^{g,[Z_1]}_{d-2}.
$$
The quadratic fundamental form at $G$ is 
$$
q_1(G,G)=r\check{r}\zeta_{2d}^{\alpha_0+1}(\zeta_{2d}^{\alpha_0}-\zeta_{2d})^2x_0^{d-4}x_1^{d-3}\{(r[\zeta_{2d}^2(d-1)+\zeta_{2d}^{\alpha_0+1}]+\check{r}[\zeta_{2d}^{2\alpha_0}(d-1)+\zeta_{2d}^{\alpha_0+1}])x_0
$$
$$
+(r[(d-1)\zeta_{2d}^2(\zeta_{2d}^{\alpha_0}+\zeta_{2d})+2\zeta_{2d}^{2\alpha_0+1}]+\check{r}[(d-1)\zeta_{2d}^{2\alpha_0}(\zeta_{2d}^{\alpha_0}+\zeta_{2d})+2\zeta_{2d}^{\alpha_0+2}])x_1\}.
$$
In order to see that this is non-zero in $R^{g}/\langle P_{Z_1}\rangle$, let us note first that $\langle P_{Z_1}\rangle_{2d-6}$ is a 2-dimensional subspace of $R^{g}_{2d-6}$. In fact, in the basis of $R^{g}_{2d-6}=\C\cdot x_0^{d-2}x_1^{d-4}\oplus \C\cdot x_0^{d-3}x_1^{d-3}\oplus\C\cdot x_0^{d-4}x_1^{d-2}$ 
$$
\langle P_{Z_1}\rangle_{2d-6}=\C\cdot Q_1\oplus \C\cdot Q_2
$$
where
$$
Q_1=(r\zeta_{2d}+\check{r}\zeta_{2d}^{\alpha_0})x_0^{d-2}x_1^{d-4}+ (r\zeta_{2d}^2+\check{r}\zeta_{2d}^{2\alpha_0})x_0^{d-3}x_1^{d-3}+ (r\zeta_{2d}^3+\check{r}\zeta_{2d}^{3\alpha_0})x_0^{d-4}x_1^{d-2},
$$
$$
Q_2=(r\zeta_{2d}^2+\check{r}\zeta_{2d}^{2\alpha_0})x_0^{d-2}x_1^{d-4}+(r\zeta_{2d}^3+\check{r}\zeta_{2d}^{3\alpha_0})x_0^{d-3}x_1^{d-3}+(r\zeta_{2d}^4+\check{r}\zeta_{2d}^{4\alpha_0})x_0^{d-4}x_1^{d-2}.
$$
Hence $q_1(G,G)$ vanishes if and only if $q_1(G,G)$, $Q_1$ and $Q_2$ are linearly dependent in $R^g_{2d-6}$. Using the monomial basis of $R^g_{2d-6}$ we can write a $3\times 3$ matrix $M$ whose columns correspond to $q_1(G,G)$, $Q_1$ and $Q_2$. Computing its determinant we obtain
$$
\det(M)=\pm\zeta_{2d}^{3\alpha_0+3}(\zeta_{2d}^{\alpha_0}-\zeta_{2d})^3r^2\check{r}^2(r-\check{r})
$$
which is non-zero for $r\neq \check{r}$.
\end{proof}

\begin{remark}
This result was already pointed out by Movasati in \cite[Theorem 18.3]{ho13} for a finite number of examples. In the case of surfaces Dan \cite{dan2021conjecture} showed that these Hodge loci are non reduced but with reduced structure equal to $V_{[\P^1]}\cap V_{[\check{\P}^1]}$ at a general surface. The higher dimensional case has also been studied recently by Kloosterman in  \cite{kloosterman2024cod1} who has shown that $V_{r[\P^\frac{n}{2}]+\check{r}[\check{\P}^\frac{n}{2}]}$ is smooth for $n\ge 4$ and coincides with $V_{[\P^\frac{n}{2}]}\cap V_{[\check{\P}^\frac{n}{2}]}$ at a general hypersurface. Kloosterman result 
 together with ours imply that $V_{r[\P^\frac{n}{2}]+\check{r}[\check\P^\frac{n}{2}]}$ is globally reducible (as a scheme), and more than one irreducible component is passing through the Fermat point.
\end{remark}

\section{Hilbert function associated to a Hodge cycle}
\label{sec6}
\cref{thm2} gives us the tensor product structure of the Artinian Gorenstein algebra associated to a join algebraic cycle. This structure helps us to study its associated Hilbert function.

\begin{definition}
\label{defHF}
Let $X=\{f=0\}\subseteq\P^{n+1}$ be a smooth degree $d$ hypersurface of even dimension $n$. For every $\lambda\in H^{\frac{n}{2},\frac{n}{2}}(X,\Q)$, its associated \textit{Hilbert function} $\HF_\lambda:\Z_{\ge 0}\rightarrow \Z_{\ge 0}$ is the Hilbert function of its associated Artinian Gorenstein algebra $R^{f,\lambda}=\C[x_0,\ldots,x_{n+1}]/J^{f,\lambda}$.
\end{definition}

\begin{cor}
\label{corHFjoin}
In the same context of \cref{thm1} we have $
\HF_{[J(Z_1,Z_2)]}=\HF_{[Z_1]}*\HF_{[Z_2]}
$, this means that for all $k\ge 0$
\begin{equation}
\label{eqHF}
\HF_{[J(Z_1,Z_2)]}(k)=\sum_{p+q=k}\HF_{[Z_1]}(p)\cdot \HF_{[Z_2]}(q).
\end{equation}
\end{cor}

\begin{proof}
This follows from \cref{thm2}.
\end{proof}

\begin{example}
Using the above corollary we can compute the Hilbert function of the examples inside Fermat described in the previous section. As an illustration for one linear cycle in Fermat (see \cref{ex:AGlincyc}) we get 
\begin{equation}
\label{eqHFlin}
\HF_{[\P^\frac{n}{2}]}=\varphi^{*(\frac{n}{2}+1)}
\end{equation}
where $\varphi:\Z_{\ge 0}\rightarrow \Z_{\ge 0}$ is the Hilbert function of a point in a $0$-dimensional Fermat variety
$$
\varphi(k)=\left\{\begin{array}{cc}
    1 & \text{ if }0\le k\le d-2, \\
    0 & \text{ otherwise}.
\end{array}\right.
$$
In other words $\HF_{[\P^\frac{n}{2}]}(k)$ counts the number of ways of writing $k$ as an ordered sum of $\frac{n}{2}+1$ numbers between $0$ and $d-2$.
\end{example}

\begin{remark}
The Hilbert function of \eqref{eqHFlin} is in fact the Hilbert function of a generic linear cycle inside a smooth degree $d$ hypersurface of even dimension $n$. This follows from the upper semi-continuity of the Hilbert function along the locus of hypersurfaces containing an $\frac{n}{2}$-dimensional linear cycle. In fact, the upper semi-continuity of the Hilbert function holds along the locus of hypersurfaces containing an $\frac{n}{2}$-dimensional complete intersection for any fixed multi-degree. This is a direct consequence of the explicit description of generators of the associated Artinian Gorenstein ideal which can be found in \cite[Example 2.1]{villaflor2022small}. In particular we can compute the Hilbert function of a generic complete intersection of type $(1,1,\ldots,1,k)$ by writing it as a join in Fermat. In general, for other types of algebraic cycles $\lambda$ we do not know whether the Hilbert function is upper semi-continuous along $V_\lambda$ or not.
\end{remark}

\section{Fake algebraic cycles}
\label{sec7}
In the article \cite{DuqueVillaflor} the authors found pathological algebraic cycles in all Fermat varieties of degree $3$, $4$ and $6$. They were pathological in the sense that their associated Hodge loci $V_\lambda$ had the biggest possible Zariski tangent space at the Fermat point without being $\lambda$ the class of a linear cycle (contradicting a conjecture of Movasati). These cycles were called fake linear cycles and were constructed from an arithmetic viewpoint using the Galois action in the cohomology of the Fermat variety. Using the join construction we can have a better understanding of these cycles as explicit combinations of linear cycles. In this section we will introduce a more general notion of fake algebraic cycles, inside any smooth hypersurface and we will show how one can find hypersurfaces containing fake linear cycles in any degree. 

\begin{definition}
\label{deffake}
Let $X=\{f=0\}\subseteq\P^{n+1}$ be a smooth degree $d$ hypersurface of even dimension $n$. Let $Z\subseteq X$ be an $\frac{n}{2}$-dimensional algebraic subvariety. A Hodge cycle $\lambda\in H^{\frac{n}{2},\frac{n}{2}}(X,\Q)$ is a \textit{fake version of $[Z]$} if 
$$
\HF_\lambda=\HF_{[Z]}
$$
but $\lambda_\prim$ is not a scalar multiple of $[Z]_\prim$.
\end{definition}

\begin{remark}
By \cite[Theorem 1.1]{DuqueVillaflor} all Fermat varieties of degree $d=3,4,6$ (and only for those degrees) admit fake linear cycles. In fact, in this case the main result shows that $\HF_\lambda(d)={\frac{n}{2}+d\choose d}-(\frac{n}{2}+1)^2=\HF_{[\P^\frac{n}{2}]}(d)$ implies 
\begin{equation}
\label{eqfakelcFermat}
P_\lambda=c_\lambda\prod_{j=0}^\frac{n}{2}\left(\frac{x_{2j}^{d-1}-(c_jx_{2j+1})^{d-1}}{x_{2j}-c_jx_{2j+1}}\right)
\end{equation}
for any $c_j\in\zeta_{2d}^{-3}\cdot \{z\in \Q(\zeta_d): \ |z|=1\}$ and some $c_\lambda\in \Q(\zeta_{2d})^\times$. From this we deduce that 
\begin{equation}
\label{eqAGfakelcFermat}
J^{f,\lambda}=\langle x_0-c_0x_1,x_2-c_1x_3,\ldots,x_n-c_\frac{n}{2}x_{n+1},x_0^{d-1},\ldots,x_{n+1}^{d-1}\rangle,
\end{equation}
which in turn implies that $\HF_\lambda=\HF_{[\P^\frac{n}{2}]}$. In the case where all $c_j$ are $d$-th roots of $-1$, $\lambda$ corresponds to the class of a linear cycle in Fermat. In all other cases $\lambda$ is a fake linear cycle. The description of the Artinian Gorenstein ideal \eqref{eqAGfakelcFermat} implies that
$$
R^{f,\lambda}=\bigotimes_{j=1}^{\frac{n}{2}+1} R^{f_j,\lambda_j}
$$
where $X_j=\{f_j(x_{2j-2},x_{2j-1}):=x_{2j-2}^d+x_{2j-1}^d=0\}\subseteq\P^1$ and $\lambda_j$ is the class of a $0$-cycle such that 
$$
P_{\lambda_j}=\frac{x_{{2j-2}}^{d-1}-(c_{2j-2}x_{{2j-1}})^{d-1}}{x_{{2j-2}}-c_{2j-2}x_{{2j-1}}}.
$$
In other words, each $\lambda_j$ is a $0$-dimensional fake linear cycle. Since this is a Hodge cycle, there exist $n_{j,1},\ldots,n_{j,d}\in\Q$ such that
$$
P_{\lambda_j}=\sum_{\ell=1}^dn_{j,\ell}\cdot P_{[p^j_\ell]}
$$
where $X_j=\{p^j_1,p^j_2,\ldots,p^j_{d}\}\subseteq\P^1$ (note that each $p^j_\ell\in \CH^0(X_j)$ is a linear cycle, and so we know how to compute $P_{[p^j_\ell]}$). It follows from \cref{thm2} that 
$$
\lambda_\prim=c\cdot J(\lambda_1,\lambda_2,\ldots,\lambda_{\frac{n}{2}+1})
$$
for some $c\in\Q^\times$.
In other words, every fake linear cycle is a linear combination of linear cycles given by
$$
\lambda=\sum_{\ell_1,\ell_2,\ldots,\ell_{\frac{n}{2}+1}=1}^d(\prod_{j=1}^{\frac{n}{2}+1}n_{j,\ell_j})\cdot J(p^1_{\ell_1},p^2_{\ell_2},\ldots,p^{\frac{n}{2}+1}_{\ell_{\frac{n}{2}+1}}).
$$
\end{remark}



\begin{remark}
The previous remark shows that the presence of fake linear cycles in degree $d=3,4,6$ Fermat varieties is due to their existence in $0$-dimensional Fermat varieties of such degrees. Using this observation one can go further and produce some fake versions of other algebraic cycles obtained as joins. For instance we can produce fake versions of complete intersection cycles of type $(1,1,\ldots,1,2)$ in Fermat varieties of degree $d=3,4,6$ by taking
$$
\lambda=J(\lambda_1,[Z_2])
$$
where $\lambda_1$ is a fake linear cycle in $X_1=\{x_0^d+\cdots+x_{n-1}^d=0\}$ and $Z_2=p_1+p_2\in \CH^0(X_2)$ for $X_2=\{x_n^d+x_{n+1}^d=0\}=\{p_1,p_2,\ldots,p_d\}$. More generally, for any algebraic cycle given as a cone 
$$
Z=J(pt,Z_2)
$$
we can construct a fake version of $Z$ if we replace the point by a $0$-dimensional fake linear cycle. Hence it is natural ask whether there are more $0$-dimensional hypersurfaces (of higher degree) containing $0$-dimensional fake linear cycles. It turns out that it is not hard to construct hypersurfaces with infinitely many fake linear cycles in any degree.
\end{remark}

\begin{theorem}
\label{thm5}
Let $X=\{f(x_0,x_1):=(x_0-r_1x_1)(x_0-r_2x_1)\cdots(x_0-r_dx_1)=0\}\subseteq\P^1$ be a smooth degree $d$ hypersurface with $r_i\in\Q$ for all $i=1,\ldots,d$. Consider for each $c\in \Q\setminus\{r_1,\ldots,r_d\}$ the polynomial
\begin{equation}
\label{eqPfake0dim}
P:=\frac{a\frac{\partial f}{\partial x_0}-b\frac{\partial f}{\partial x_1}}{x_0-cx_1}\in R^f_{d-2}
\end{equation}
for $a=\frac{\partial f}{\partial x_1}(c,1)$ and $b=\frac{\partial f}{\partial x_0}(c,1)$. Then 
\begin{equation}
\label{eqfake0dim}
\delta:=\res\left(\frac{P\cdot(x_0dx_1-x_1dx_0)}{f}\right)\in H^0(X,\Q)_\prim
\end{equation}
is a $0$-dimensional fake linear cycle.
\end{theorem}

\begin{proof}
For each point $p_i:=(r_i:1)\in X$ we know $[p_i]_\prim\in H^0(X,\Q)$. Moreover we can write it as a residue applying \cite[Theorem 1.1]{villaflor2021periods}
$$
[p_i]_\prim=\frac{-1}{d}\res\left(\frac{P_i\cdot(x_0dx_1-x_1dx_0)}{f}\right)\in H^0(X,\Q)_\prim
$$
where 
\begin{equation}
\label{eqPi}
P_i=\det\begin{pmatrix}
    1 & \frac{\frac{\partial f}{\partial x_0}}{x_0-r_ix_1}-\frac{f}{(x_0-r_ix_1)^2} \\
    -r_i & \frac{\frac{\partial f}{\partial x_1}}{x_0-r_ix_1}+r_i\frac{f}{(x_0-r_ix_1)^2}
\end{pmatrix}=\frac{r_i\frac{\partial f}{\partial x_0}+\frac{\partial f}{\partial x_1}}{x_0-r_ix_1}\in \Q[x_0,x_1]_{d-2}.
\end{equation}
Since all the points $[p_1]_\prim,\ldots,[p_d]_\prim$ generate the $\Q$-vector space $H^0(X,\Q)_\prim$ of dimension $d-1$, and the residue map is an isomorphism of $\C[x_0,x_1]_{d-2}=R^f_{d-2}\simeq H^0(X,\C)_\prim$, it follows that the polynomials $P_1,\ldots,P_d$ generate all $\Q[x_0,x_1]_{d-2}$ as $\Q$-vector space. In particular, since $c\in \Q$, then $P\in\Q[x_0,x_1]_{d-2}$ and so we can write it as a $\Q$-linear combination 
$$
P=q_1\cdot P_1+\cdots+q_d\cdot P_d.
$$
Hence $\delta=q_1\cdot [p_1]_\prim+\cdots+q_d\cdot [p_d]_\prim\in H^0(X,\Q)_\prim$ is a rational class. To see that it defines a fake linear cycle it is enough to see that 
$$
J^{f,\delta}=(J^f:P)=\langle x_0-cx_1,x_0^{d-1},x_1^{d-1}\rangle
$$
and so $\HF_\delta=\HF_{[p_i]}$.
\end{proof}

Now, as a corollary of \cref{thm5} we obtain \cref{thm3}.

\bigskip

\begin{proof}\textbf{of \cref{thm3}}
Pick any degree $d$ homogeneous polynomials $f_0,\ldots,f_{\frac{n}{2}}\in\Q[x,y]_d$ such that each $f_i$ has only simple rational roots. Define $X:=\{f_0(x_0,x_1)+f_1(x_2,x_3)+\cdots+f_\frac{n}{2}(x_n,x_{n+1})=0\}\subseteq\P^{n+1}$. For each $i=0,\ldots,\frac{n}{2}$ consider $X_i:=\{f_i(x_{2i},x_{2i+1})=0\}\subseteq\P^1$ and take any fake linear cycle $\delta_i\in H^0(X_i,\Q)_\prim$. Then by \cref{corHFjoin} 
$$
\delta:=J(\delta_0,\ldots,\delta_\frac{n}{2})\in H^{\frac{n}{2},\frac{n}{2}}(X,\Q)_\prim
$$
is a fake linear cycle.
\end{proof}

\begin{remark}
A consequence of \cref{thm5} and \cite[Theorem 1.1]{villaflor2022small} is that no automorphism of $\P^1$ transforms all points of the Fermat variety $X=\{x_0^d+x_1^d=0\}\subseteq\P^1$ into rational points for $d\neq 3,4,6$. On the other hand, it is easy to check that for degrees $d=3,4,6$ there exists an automorphism of $\P^1$ taking all Fermat points to rational points. This explains the presence of fake linear cycles in Fermat varieties of such degrees.
\end{remark}

\begin{example}
Let $X=\{f(x_0,x_1):=(x_0-r_1x_1)(x_0-r_2x_1)\cdots(x_0-r_6x_1)=0\}\subseteq\P^1$ with $r_1=0,$ $r_2=1,$ $r_3=\frac{1}{2},$ $r_4=\frac{1}{4},$ $r_5=\frac{1}{3},$ $r_6=\frac{2}{5}.$ Consider the same notation of \cref{thm5}, and take the fake linear cycle $\delta$ of the form \eqref{eqfake0dim} where the polynomial $P$ in \eqref{eqPfake0dim} is defined using the number $c=-1$. Let $P_i$ be the associated polynomial to the point $(r_i:1)\in X$ for $i=1,\ldots,5$ (this is computed explicitly in \eqref{eqPi}). Once we know explicitly all these polynomials, it is an elementary linear algebra problem to find the $\Q$-linear combination of the polynomial $P$ in terms of the polynomials $P_1,\ldots,P_5$, which is 
$$
P=-\frac{207283}{810}P_1-\frac{68941}{270}P_2-\frac{507311}{1620}P_3-\frac{26911}{180}P_4-\frac{891881}{1620}P_5.
$$
For the case of fake linear cycles in Fermat varieties of degree $d=3,4,6$ one first transforms the Fermat equation to one with only rational roots, and proceeds in the same way as before. In fact, the above example is isomorphic to the Fermat sextic under the composition of the following automorphisms of $\P^1$
$$\phi(x_0:x_1)=(x_0:x_0+x_1)$$
$$
\psi(x_0:x_1)=(x_0-\zeta_{12}x_1:(1+\zeta_6^{-1})(x_0-\zeta_{12}^3x_1)).
$$
We have that $\psi^* \phi^*(f)=-\frac{2\zeta_6^2+1}{40}(x_0^6+x_1^6)$ and the $0$-dimensional fake linear cycle $\delta\in H^0(X,\Q)_\prim$ is transformed to the fake linear cycle $\lambda=\psi^*\phi^*\delta$ inside the Fermat variety $\{x_0^6+x_1^6=0\}\subset\P^1$  given by
$$
\lambda=\res\left(\frac{P_\lambda\cdot(x_0dx_1-x_1dx_0)}{x_0^6+x_1^6}\right)
$$
with 
$$
P_\lambda=c_\lambda\frac{x_0^5-(c_0x_1)^5}{x_0-c_0x_1},
$$
where
$c_0=\zeta_{12}^{-3}\left(\frac{3\zeta_6^2-1}{3-\zeta_6^2}\right)\in \zeta_{12}^{-3}\cdot\mathbb{S}^1_{\Q(\zeta_6)}$ and $c_\lambda=\frac{\zeta_{12}(191\zeta_6+146)}{1-2\zeta_6}\in\Q(\zeta_{12})^\times.$
\end{example}

As a final result of this section we show the non smoothness of the Hodge loci associated to fake linear cycles.

\begin{theorem}
\label{thm6}
Let $n$ an even number and $d\ge 2+\frac{6}{n}$ an integer. For any degree $d$ homogeneous polynomials $f_0,\ldots,f_\frac{n}{2}\in\Q[x,y]_d$ with no multiple roots, let $$X=\{f_0(x_0,x_1)+f_1(x_2,x_3)+\cdots+f_\frac{n}{2}(x_n,x_{n+1})=0\}\subseteq\P^{n+1}.$$ 
For each $i=0,\ldots,\frac{n}{2}$ consider $X_i:=\{f_i(x_{2i},x_{2i+1})=0\}\subseteq\P^1$ and some $\delta_i\in H^0(X_i,\Q)_\prim$ with $\HF_{\delta_i}$ equal to the Hilbert function of a point. Let $\delta:=J(\delta_1,\ldots,\delta_\frac{n}{2})$, then $\delta$ is a fake linear cycle if and only if $V_\delta$ is not non-reduced or is reduced and singular at $X$.
\end{theorem}

\begin{proof}
If all $\delta_i$ are the primitive classes of points (up to scalar multiplication), then $\delta$ is the primitive class of a linear cycle (up to scalar multiplication), and so $V_\delta$ is known to be smooth. If some $\delta_i$ is a $0$-dimensional fake linear cycle, then we can write (up to scalar multiplication) $\delta_i=\res\left(\frac{P\cdot(x_{2i}dx_{2i+1}-x_{2i+1}dx_{2i})}{f_i}\right)$ for
$$
P=\frac{a\frac{\partial f_i}{\partial x_{2i}}-b\frac{\partial f_i}{\partial x_{2i+1}}}{x_{2i}-cx_{2i+1}}
$$
where $a=\frac{\partial f_i}{\partial x_{2i+1}}(c,1)$, $b=\frac{\partial f_i}{\partial x_{2i}}(c,1)$ and $f_i(c,1)\neq 0$. If we compute the quadratic fundamental form $q_i$ of $V_{\delta_i}$ at the term $x_{2i}-cx_{2i+1}\in J^{f_i,\delta_i}_1$ we get
$$
q_i(x_{2i}-cx_{2i+1},x_{2i}-cx_{2i+1})=a+bc=d\cdot f_i(c,1)\neq0.
$$
We claim that $R^{f_i}/\langle P\rangle$ is non-zero at degree $2d-5$. In fact, since $J^{f_i}$ is Artinian Gorenstein of socle in degree $2d-4$, then $(J^{f_i}:x_{2i}-cx_{2i+1})$ is Artinian Gorenstein of socle in degree $2d-5$ and so there exists some $Q\in \C[x_{2i},x_{2i+1}]_{2d-5}$  such that $Q\cdot (x_{2i}-cx_{2i+1})\notin J^{f_i}$. But since $P\cdot (x_{2i}-cx_{2i+1})\in J^{f_i}$, it follows that $Q\notin \langle P\rangle$, and so $(R^{f_i}/\langle P\rangle)_{2d-5}\neq 0$ as claimed. Therefore
$$
q_i|_{\text{Sym}^2(J^{f_i,\delta_i}_1)}\cdot \C[x_{2i},x_{2i+1}]_{2d-5}\neq 0\in R^{f_i}/\langle P\rangle.
$$
Using that $d\ge 2+\frac{6}{n}$ we can apply \cref{corqff} (ii) with the values $e=d$, $\ell=1$, $j=2d-5$, $k=0$ to conclude that the quadratic fundamental form $q$ associated to $V_\delta$ is non-zero in degree $d$, and so $V_\delta$ is not smooth.
\end{proof}






\bibliographystyle{abbrv}

\bibliography{ref}

\begin{thebibliography}{10}

\bibitem{carlson1983infinitesimal}
J.~Carlson, M.~Green, P.~Griffiths, and J.~Harris.
\newblock Infinitesimal variations of {H}odge structure (i).
\newblock {\em Compositio Mathematica}, 50(2-3):109--205, 1983.

\bibitem{dan2021conjecture}
A.~Dan.
\newblock On a conjecture of {H}arris.
\newblock {\em Communications in Contemporary Mathematics}, 23(07):2050028, 2021.

\bibitem{DuqueVillaflor}
J.~Duque~Franco and R.~Villaflor~Loyola.
\newblock On fake linear cycles inside {F}ermat varieties.
\newblock {\em Algebra \& Number Theory}, 17(10):1847--1865, 2023.

\bibitem{green1984koszul}
M.~L. Green.
\newblock Koszul cohomology and the geometry of projective varieties. ii.
\newblock {\em Journal of Differential Geometry}, 20(1):279--289, 1984.

\bibitem{green1988}
M.~L. Green.
\newblock A new proof of the explicit {N}oether-{L}efschetz theorem.
\newblock {\em J. Differential Geom.}, 27(1):155--159, 1988.

\bibitem{kloosterman2024cod1}
R.~Kloosterman.
\newblock Hodge loci associated with linear subspaces intersecting in codimension one.
\newblock {\em Mathematische Nachrichten}, 298(4):1220--1229, 2025.

\bibitem{kloosterman2023conjecture}
R.~Kloosterman.
\newblock On a conjecture on hodge loci of linear combinations of linear subvarieties.
\newblock {\em Rendiconti del Circolo Matematico di Palermo Series 2}, 74(6):187, 2025.

\bibitem{mclean2005}
C.~Maclean.
\newblock A second-order invariant of the {N}oether-{L}efschetz locus and two applications.
\newblock {\em Asian Journal of Mathematics}, 9(3):373--400, 2005.

\bibitem{GMCD-NL}
H.~Movasati.
\newblock Gauss-{M}anin connection in disguise: {N}oether-{L}efschetz and {H}odge loci.
\newblock {\em Asian Journal of Mathematics}, 2017.

\bibitem{ho13}
H.~Movasati.
\newblock {\em A course in {H}odge theory: with emphasis on multiple integrals}.
\newblock International Press of Boston, 2021.

\bibitem{movasati2023hodge}
H.~Movasati.
\newblock On a {H}odge locus.
\newblock {\em arXiv preprint arXiv:2211.11405}, 2022.

\bibitem{MV}
H.~Movasati and R.~Villaflor.
\newblock Periods of linear algebraic cycles.
\newblock {\em Pure and Applied Mathematics Quarterly}, 14, 04 2018.

\bibitem{Otw02}
A.~Otwinowska.
\newblock Sur la fonction de hilbert des algèbres graduées de dimension 0.
\newblock {\em Journal Fur Die Reine Und Angewandte Mathematik - J REINE ANGEW MATH}, 2002:97--119, 01 2002.

\bibitem{Otwinowska2003}
A.~Otwinowska.
\newblock Composantes de petite codimension du lieu de {N}oether-{L}efschetz: un argument asymptotique en faveur de la conjecture de {H}odge pour les hypersurfaces.
\newblock {\em J. Algebraic Geom.}, 12(2):307--320, 2003.

\bibitem{sebastiani1971resultat}
M.~Sebastiani and R.~Thom.
\newblock Un r{\'e}sultat sur la monodromie.
\newblock {\em Inventiones mathematicae}, 13:90--96, 1971.

\bibitem{villaflor2021periods}
R.~Villaflor~Loyola.
\newblock Periods of complete intersection algebraic cycles.
\newblock {\em manuscripta mathematica}, 167(3-4):765--792, 2022.

\bibitem{villaflor2022small}
R.~Villaflor~Loyola.
\newblock Small codimension components of the {H}odge locus containing the {F}ermat variety.
\newblock {\em Communications in Contemporary Mathematics}, 24(07), 2022.

\bibitem{villaflor2023toric}
R.~Villaflor~Loyola.
\newblock Toric differential forms and periods of complete intersections.
\newblock {\em Journal of Algebra}, 643:86--118, 2024.

\bibitem{voisin1988}
C.~Voisin.
\newblock Une pr\'ecision concernant le th\'eor\`eme de {N}oether.
\newblock {\em Math. Ann.}, 280(4):605--611, 1988.

\bibitem{voisin89}
C.~Voisin.
\newblock Composantes de petite codimension du lieu de {N}oether-{L}efschetz.
\newblock {\em Comment. Math. Helv.}, 64(4):515--526, 1989.

\bibitem{voisin2003hodge}
C.~Voisin.
\newblock {\em Hodge Theory and Complex Algebraic Geometry II: Volume 2}, volume~77.
\newblock Cambridge University Press, 2003.

\end{thebibliography}



\bigskip

\noindent{\sc Dirección de Investigación, Vicerrectoría Académica \\
Instituto de Matemática y Fisíca, Universidad de Talca}\\
{\sc Avenida Lircay s/n, Casilla 721, Talca, Chile}\\
\textit{Email address:} {\tt georgy11235@gmail.com}\\
\\
\noindent{\sc Departamento de Matemática, Universidad Técnica Federico Santa María}\\
{\sc Avenida España 1680, Valparaíso, Chile}\\
\textit{Email address:} {\tt roberto.villaflor@usm.cl}

\end{document}